\documentclass[11pt,reqno,a4paper,oneside,11pt]{amsart}%
\usepackage{amsfonts}
\usepackage{mathrsfs}
\usepackage{amssymb}
\usepackage{amsmath}
\usepackage{amsthm}
\usepackage{graphicx}
\usepackage{color}
\usepackage{extarrows}

\providecommand{\U}[1]{\protect\rule{.1in}{.1in}}
\newtheorem{theorem}{Theorem}[section]
\newtheorem{corollary}[theorem]{Corollary}
\newtheorem{lemma}[theorem]{Lemma}

\newtheorem{remark}[theorem]{Remark}
\newtheorem{algorithm}[theorem]{Algorithm}
\newtheorem{example}[theorem]{Example}

\theoremstyle{definition}
\theoremstyle{remark}
\numberwithin{equation}{section}

\ifx\pdfoutput\relax\let\pdfoutput=\undefined\fi
\newcount\msipdfoutput
\ifx\pdfoutput\undefined\else
\ifcase\pdfoutput\else
\ifx\paperwidth\undefined\else
\ifdim\paperheight=0pt\relax\else\pdfpageheight\paperheight\fi
\ifdim\paperwidth=0pt\relax\else\pdfpagewidth\paperwidth\fi
\fi\fi\fi
\begin{document}
\pagestyle{myheadings}

\begin{center}
{\huge \textbf{  A Sylvester-type matrix equation over the Hamilton quaternions with an application}}\footnote{This research was supported by National Natural
Science Foundation of China grant (11971294 and 12171369)
\par
{}* Corresponding author. \par  Email address: wqw@t.shu.edu.cn, wqw369@yahoo.com (Q.W. Wang).}

\bigskip

{ \textbf{Long-Sheng Liu$^{a}$, Qing-Wen Wang$^{a, b *}$, Mahmoud Saad  Mehany}$^{a, c}$}\\
{\small
\vspace{0.25cm}
 a Department of Mathematics, Shanghai University, Shanghai 200444, P. R. China\\
 b Collaborative Innovation Center for the Marine Artificial Intelligence, Shanghai 200444, P. R. China\\
c Department of Mathematics, Ain Shams University, Cairo, 11566, A.R. Egypt \quad\quad}
\end{center}
\vspace{1cm}
\begin{quotation}
\noindent\textbf{Abstract:}
 We derive the solvability conditions and a formula of a general solution to a Sylvester-type matrix equation over Hamilton quaternions. As an application, we investigate the necessary and sufficient conditions for the solvability of the quaternion matrix equation,
which involves $\eta$-Hermicity. 
We also provide an algorithm with a numerical example to illustrate the main results of this paper.

\vspace{3mm}

\end{quotation}
\section{\textbf{Introduction}}
Let $\mathbb{R}$ stand for the real number field and  
$$\mathbb{H}=\{u_{0}+u_{1}\mathbf{i}+u_{2}\mathbf{j}+u_{3}\mathbf{k} | \mathbf{i}^{2}=\mathbf{j}^{2}=\mathbf{k}^{2}=\mathbf{ijk}=-1,\ u_{0}, u_{1}, u_{2}, u_{3}\in \mathbb{R}\}.$$  $\mathbb{H} $ is called the Hamilton quaternion algebra, which is a non-commutative division ring.  Hamilton quaternions and Hermitian quaternion matrices have been utilized in statistics of quaternion random signals \cite{C.C 2011}, quaternion matrix optimization problems \cite{L.Q. 2021}, signal and color image processing,
face recognition  \cite{  Jia 2018, Wang X.X.},
and so on. 

Sylvester and Sylvester-type matrix equations have a large number of applications in different disciplines and fields. For example, the Sylvester matrix equation
\begin{align}\label{eq00}
A_{1}X+XB_{1}=C_{1}
\end{align}
and the Sylvester-type matrix equation
\begin{align}\label{eq2}
A_{1}X+YB_{1}=C_{1}
\end{align}
have been applied in singular system control \cite{Shahzad}, system design \cite{Syrmosw33}, perturbation theory \cite{Liw33}, sensitivity analysis \cite{Barraud}, $H_{\alpha}$-optimal control \cite{Saberiw33}, linear descriptor systems \cite{Darouach},
and control theory \cite{E.B. 2005}.   Roth \cite{W.E. 1952} gave the Sylvester-type matrix Equation \eqref{eq2} for the first time over the polynomial integral domain. Baksalary and Kala \cite{J.K. 1979} established the solvability conditions for Equation \eqref{eq2}  and gave an expression of its general solution. In addition, Baksalary and Kala  \cite{J.K. 1980} derived the necessary and sufficient  conditions for a two-sided Sylvester-type matrix equation
\begin{align}\label{eq3}
A_{11}X_{1}B_{11}+C_{11}X_{2}D_{11}=E_{11}
\end{align}
to be consistent. \"{O}zg\"{u}ler \cite{A.B. 1991} studied \eqref{eq3} over a principal ideal domain. Wang \cite{T28} investigated \eqref{eq3} over an arbitrary regular ring with an identity element. 

Due to the wide applications of quaternions, the investigations on Sylvester-type matrix equations have been extended to $\mathbb{H}$ in
the last decade (see, e.g., \cite{Jiang 2022, Liu,  WangThree,  xinliu01, xinliu03,  xinliu04, Wang 2019, Wang 2016}). 
They are applied for signal processing, color-image processing, and maximal invariant semidefinite or neutral subspaces,
etc. (see, e.g., \cite{ Rodman, Jia 2019,  He 2021,  S.F. 2013}). 
For instance, the general solution to Sylvester-type matrix Equation \eqref{eq2} can be used 
in color-image processing.
He \cite{He 2019} derived the matrix Equation \eqref{eq2} as an essential finding.
Roman \cite{Rodman} established the necessary and sufficient conditions for Equation \eqref{eq00} to have a solution. 
Kychei \cite{Kyrchei0112} investigated Cramer's rules to drive the necessary and sufficient  conditions for Equation \eqref{eq3} to be solvable. 
As an extension of Equations
\eqref{eq2} and \eqref{eq3}, Wang and He \cite{Wang 2012} gave the solvability conditions and the general solution to the Sylvester-type matrix equation
\begin{align}\label{eq4}
A_{1}X_{1}+X_{2}B_{1}+C_{3}X_{3}D_{3}+C_{4}X_{4}D_{4}=E_{1}
\end{align}
over the complex number field 
$\mathbb{C}$, which can be generalized to $\mathbb{H}$ and applicable in some Sylvester-type matrix equations over $\mathbb{H}$ (see, e.g., \cite{ wangronghao,  He 2019}). 

We know that in system and control theory, the more unknown matrices that a matrix equation has, the wider its application will be. Consequently, for the sake of developing theoretical studies and the applications mentioned above of Sylvester-type matrix equation  and their generalizations,  in this paper,
we aim to establish some necessary and sufficient conditions for the Sylvester-type matrix equation
\begin{align}\label{eq1}
A_{1}X_{1}+X_{2}B_{1}+A_{2}Y_{1}B_{2}+A_{3}Y_{2}B_{3}+A_{4}Y_{3}B_{4}=B
\end{align}
to have a solution in terms of the rank equalities and Moore--Penrose inverses of some coefficient quaternion matrices in Equation \eqref{eq1} over $\mathbb{H}.$  
We derive a formula of its general solution when it is solvable.  It is clear that Equation \eqref{eq1} provides a proper generalization of 
Equation \eqref{eq4}, and we carry out an algorithm with a numerical example to calculate the general solution of Equation \eqref{eq1}. As a special case of Equation \eqref{eq1}, we also obtain the solvability conditions and the general solution   for the two-sided Sylvester-type matrix equation
\begin{align}\label{eq6}
A_{11}Y_{1}B_{11}+A_{22}Y_{2}B_{22}+A_{33}Y_{3}B_{33}=T_{1}.
\end{align}

To the best of our knowledge, so far, 
there has been little information on the solvability conditions and an expression of the general solution to Equation \eqref{eq6} by using generalized inverses.

As usual, we use $A^{\ast}$ to denote the conjugate transpose of $A$. Recall that a quaternion matrix $A$,  for\ ${\eta}\in \{\mathbf{i},\mathbf{j},\mathbf{k}\}$, is said to be $\eta$-Hermitian if $A=A^{\eta^{\ast}}$, where $A^{\eta^{\ast}}=-\eta A^{\ast}\eta$~\cite{ F.Z 2011}. For more properties and information on  $\eta^{*}$-quaternion matrices, 
we refer to \cite {F.Z 2011}. 
We know that $\eta$-Hermitian matrices have some applications in linear modeling and statistics of quaternion random signals \cite{C.C 2011,F.Z 2011}. As an application of Equation \eqref{eq1}, we establish some necessary and sufficient conditions for the  quaternion matrix equation
\begin{equation}\label{eq5}
\begin{aligned}
&A_{1}X_{1}+(A_{1}X_{1})^{\eta^{\ast}}+A_{2}Y_{1}A_{2}^{\eta^{\ast}}+A_{3}Y_{2}A_{3}^{\eta^{\ast}}+A_{4}Y_{3}A_{4}^{\eta^{\ast}}=B
\end{aligned}
\end{equation}
to be consistent. Moreover, we derive a formula of the general solution to Equation \eqref{eq5} where $B=B^{\eta^{\ast}}, Y_{i}=Y_{i}^{\eta^{\ast}}$ $(i=\overline{1,3})$ over $\mathbb{H}$.

The rest of this paper is organized as follows. In Section \ref{sec2}, we review some definitions and lemmas. In Section  \ref{sec3}, we establish some necessary and sufficient conditions for Equation \eqref{eq1} to have a solution. In addition, we give an expression of its general solution to Equation \eqref{eq1} when it is solvable. In Section  \ref{sec4}, as an application of 
Equation \eqref{eq1},  we consider some solvability conditions and the general  solution to Equation \eqref{eq5}, where  $Y_{i}=Y_{i}^{\eta^{\ast}}$ $(i=\overline{1,3})$. Finally, we give a brief conclusion to the paper in Section  \ref{sec5}.
\section{Preliminaries \label{sec2}}

Throughout this paper, $\mathbb{H}^{m\times n}$ stands for the space of all $m\times n$ matrices  over $\mathbb{H}$. The symbol $r(A)$ denotes the rank of $A$. $I$ and $0$ represent an identity matrix and a zero matrix of appropriate sizes, respectively. In general, $A^{\dagger}$ stands for the Moore--Penrose inverse of $A\in \mathbb{H}^{l\times k}$,  which is defined as the solution of $AYA=A, \  YAY=Y,\  (AY)^{\ast}=AY$ and $(YA)^{\ast}=YA.$ Moreover, $L_{A}=I-A^{\dagger}A$ and $R_{A}=I-AA^{\dagger}$ represent two projectors along~$A$.  

The following lemma is due to Marsaglia and Styan \cite{M.G 1974}, which can be generalized to~$\mathbb{H}$.
\begin{lemma}[\cite{M.G 1974}]\label{lem2}
	Let $A \in \mathbb{H}^{m\times n},$ $B \in \mathbb{H}^{m\times k}, $ $C \in \mathbb{H}^{l\times n}$, $D\in \mathbb{H}^{j\times k}$ and $E \in\mathbb{H}^{l\times i}$ be given. 
	Then,
	we have the following rank equality:
	$$\ r\begin{pmatrix}
	A&BL_{D}\\
	R_{E}C&0
	\end{pmatrix}=r\begin{pmatrix}
	A&B&0\\
	C&0&E\\
	0&D&0
	\end{pmatrix}-r(D)-r(E).
	$$
\end{lemma}
\begin{lemma}[\cite{Z.H. 2013}]\label{lemnew}
	Let $A \in \mathbb{H}^{m\times n}$ be given. 
	Then,
	\begin{align*}
	&\mathrm{(1)}\ \ \ (A^{\eta})^{\dagger}=(A^{\dagger})^{\eta},\ (A^{\eta^{\ast}})^{\dagger}=(A^{\dagger})^{\eta^{\ast}}.\ \\
	&\mathrm{(2)}\ \ \ r(A)=r(A^{\eta^{\ast}})=r(A^{\eta})=r(A^{\eta}A^{\eta^{\ast}})=r(A^{\eta^{\ast}}A^{\eta}).\\
	&\mathrm{(3)}\ \ \ (L_{A})^{\eta^{\ast}}=-\eta(L_{A})\eta=(L_A)^{\eta}=L_{A^{\ast}}=R_{A^{\eta^{\ast}}}.\\
	&\mathrm{(4)}\ \ \ (R_{A})^{\eta^{\ast}}=-\eta(R_{A})\eta=(R_A)^{\eta}=R_{A^{\ast}}=L_{A^{\eta^{\ast}}}.\\
	&\mathrm{(5)}\ \ \ (AA^{\dagger})^{\eta^{\ast}}=(A^{\dagger})^{\eta^{\ast}}A^{\eta^{\ast}}=(A^{\dagger}A)^{\eta}=A^{\eta}(A^{\dagger})^{\eta}.\\
	&\mathrm{(6)}\ \ \ (A^{\dagger}A)^{\eta^{\ast}}=A^{\eta^{\ast}}(A^{\dagger})^{\eta^{\ast}}=(AA^{\dagger})^{\eta}=(A^{\dagger})^{\eta}A^{\eta}.
	\end{align*}
\end{lemma}

\begin{lemma}[\cite{T28}]\label{lem2.2}  
	Let $A_{ii}, B_{ii}$ and $C_{i}\ (i=1,2)$ be given matrices with suitable sizes over $\mathbb{H}$. $A_{1}=A_{22}L_{A_{11}},\ T=R_{B_{11}}B_{22}$, $F=B_{2
		2}L_{T},\ G=R_{A_{1}}A_{22}$. 
	Then,
	the following statements are equivalent:
	
	$\mathrm{(1)}$ The system 
	\begin{equation}\label{eq7}
	\begin{aligned}
	A_{11}X_{1}B_{11}=C_{1}, \ A_{22}X_{1}B_{22}=C_{2}
	\end{aligned}
	\end{equation}
	has a solution.
	
	$\mathrm{(2)}$
	$$A_{ii}A_{ii}^{\dagger}C_{i}B_{ii}^{\dagger}B_{ii}=C_{i}\ (i=1,2)$$
	and
	$$G(A_{22}^{\dagger}C_{2}B_{22}^{\dagger}-A_{11}^{\dagger}C_{1}B_{11}^{\dagger})F=0.$$
	
	$\mathrm{(3)}$
	\begin{align*}
	&r
	\begin{pmatrix}
	A_{ii} & C_{i} \\
	\end{pmatrix}
	=r(A_{ii}),\ r
	\begin{pmatrix}
	C_{i} \\
	B_{ii} \\
	\end{pmatrix}
	=r(B_{ii})\ (i=1,2),\\
	&r
	\begin{pmatrix}
	A_{11} & C_{1} & 0 \\
	A_{22} & 0 & -C_{2} \\
	0 & B_{11} & B_{22} \\
	\end{pmatrix}
	=r
	\begin{pmatrix}
	A_{11}\\
	A_{22} \\
	\end{pmatrix}
	+r(B_{11},\ B_{22}).
	\end{align*}
\end{lemma}
\begin{lemma}[\cite{J.K. 1979}]\label{lem2.3}  
	Let $A_{1}$, $B_{1}$ and $C_{1}$ be given matrices with suitable sizes.
	Then,
	the Sylvester-type Equation \eqref{eq2} is solvable if and only if
	$$R_{A_{1}}C_{1}L_{B_{1}}=0.$$
	In this case, the general solution to Equation \eqref{eq2} can be expressed as
	\begin{align*}
	&X=A_{1}^{\dagger}C_{1}-A_{1}^{\dagger}U_{1}B_{1}+L_{A_{1}}U_{2},\
	Y=R_{A_{1}}C_{1}B_{1}^{\dagger}+A_{1}A_{1}^{\dagger}U_{1}+U_{3}R_{B_{1}},
	\end{align*}
	where $U_{1}, U_{2}$,
	and $U_{3}$ are arbitrary matrices with appropriate sizes.
\end{lemma}
\begin{lemma}[\cite{Wang 2012}]\label{lem2.5} 
	Let $A_{1}, B_{1}, C_{3}, D_{3}, C_{4}, D_{4}$ and $E_{1}$ be given matrices  over $\mathbb{H}$. Put
	\begin{align*}
	&A=R_{A_{1}}C_{3}, \ B=D_{3}L_{B_{1}},\ C=R_{A_{1}}C_{4},\ D=D_{4}L_{B_{1}},\\
	&E=R_{A_{1}}E_{1}L_{B_{1}}, \ M=R_{A}C,\ N=DL_{B},\ S=CL_{M}.
	\end{align*}
	Then,
	the following statements are equivalent:
	
	$\mathrm{(1)}$ Equation \eqref{eq4} has a solution.
	
	$\mathrm{(2)}$
	\begin{align*}
	R_{M}R_{A}E=0,\ EL_{B}L_{N}=0,\ R_{A}EL_{D}=0,\ R_EL_{B}=0.
	\end{align*}
	
	$\mathrm{(3)}$
	\begin{align*}
	&r
	\begin{pmatrix}
	E_{1} & C_{4} & C_{3} & A_{1} \\
	B_{1} & 0 & 0 & 0 \\
	\end{pmatrix}=r(B_{1})+r(C_{4},\ C_{3},\ A_{1}),\\
	&r
	\begin{pmatrix}
	E_{1} & A_{1} \\
	D_{3} & 0 \\
	D_{4} & 0 \\
	B_{1} & 0 \\
	\end{pmatrix}=r
	\begin{pmatrix}
	D_{3} \\
	D_{4} \\
	B_{1} \\
	\end{pmatrix}+r(A_{1}),\\
	&r
	\begin{pmatrix}
	E_{1} & C_{3} & A_{1} \\
	D_{4} & 0 & 0 \\
	B_{1} & 0 & 0 \\
	\end{pmatrix}
	=r(A_{1},\ C_3)+r
	\begin{pmatrix}
	D_{4} \\
	B_{1} \\
	\end{pmatrix}
	,\\
	&r
	\begin{pmatrix}
	E_{1} & C_{4} & A_{1} \\
	D_{3} & 0 & 0 \\
	B_{1} & 0 & 0 \\
	\end{pmatrix}=r(A_{1}\quad C_4)+r
	\begin{pmatrix}
	D_{3} \\
	B_{1} \\
	\end{pmatrix}.
	\end{align*}
	In this case, the general solution to Equation \eqref{eq4} can be expressed as
	\begin{align*}
	&X_{1}=A_{1}^{\dagger}(E_{1}-C_{3}X_{3}D_{3}-C_{4}X_{4}D_{4})-A_{1}^{\dagger}T_{7}B_{1}+L_{A_{1}}T_{6},\\
	&X_{2}=R_{A_{1}}(E_{1}-C_{3}X_{3}D_{3}-C_{4}X_{4}D_{4})B_{1}^{\dagger}+A_{1}A_{1}^{\dagger}T_{7}+T_{8}R_{B_{1}},\\
	&X_{3}=A^{\dagger}EB^{\dagger}-A^{\dagger}CM^{\dagger}EB^{\dagger}-A^{\dagger}SC^{\dagger}EN^{\dagger}DB^{\dagger}-A^{\dagger}ST_{2}R_{N}DB^{\dagger}+L_{A}T_{4}+T_{5}R_{B},\\
	&X_{4}=M^{\dagger}ED^{\dagger}+S^{\dagger}SC^{\dagger}EN^{\dagger}+L_{M}L_{S}T_{1}+L_{M}T_{2}R_{N}+T_{3}R_{D},
	\end{align*}
	where $T_{1},...,T_{8}$ are arbitrary matrices with appropriate sizes over $\mathbb{H}$.
\end{lemma}

\section{Some Solvability Conditions and a Formula of the General Solution  \label{sec3}}

In this section, we establish the solvability conditions and a formula of the general solution to Equation \eqref{eq1}. We begin with the following  lemma, which is used to reach the main results of this paper.
\begin{lemma}\label{lem3.1}
	Let $A_{11}, B_{11}$, $C_{11}$,
	and $D_{11}$ be given matrices with suitable sizes over $\mathbb{H}$, $A_{11}L_{A_{22}}=0$ and $R_{B_{11}}B_{22}=0$.
	Set
	\begin{equation}\label{2}
	\begin{aligned}
	&A_{1}=A_{22}L_{A_{11}},\ C_{11}=C_{2}-A_{22}A_{11}^{\dagger}C_{1}B_{11}^{\dagger}B_{22}.
	\end{aligned}
	\end{equation}
	Then, the following statements are equivalent:
	
	$\mathrm{(1)}$ The system \eqref{eq7} is consistent.
	
	$\mathrm{(2)}$
	\begin{align*}
	R_{A_{ii}}C_{i}=0,\ C_{i}L_{B_{ii}}=0\ (i=1,2), \ R_{A_{1}}C_{11}=0.
	\end{align*}
	
	$\mathrm{(3)}$
	\begin{align*}
	A_{ii}A_{ii}^{\dagger}C_{i}B_{ii}^{\dagger}B_{ii}=C_{i} \ (i=1,2), \ C_{1}B_{11}^{\dagger}B_{22}=A_{11}A_{22}^{\dagger}C_{2}.
	\end{align*}
	
	$\mathrm{(4)}$
	\begin{align*}
	&r(A_{ii}, C_{i})=r(A_{ii}), r
	\begin{pmatrix}
	B_{ii} \\
	C_{i} \\
	\end{pmatrix}
	=r(B_{ii})\  (i=1,2),\\
	&r
	\begin{pmatrix}
	C_{1} & 0 & A_{11} \\
	0 & -C_{2} & A_{22} \\
	B_{11} & B_{22} & 0 \\
	\end{pmatrix}
	=r
	\begin{pmatrix}
	A_{22}
	\end{pmatrix}
	+r(B_{11}).
	\end{align*}
	In this case, the general solution to system \eqref{eq7} can be expressed as
	\begin{equation}\label{eq8}
	\begin{aligned}
	X_{1}=A_{11}^{\dagger}C_{1}B_{11}^{\dagger}+L_{A_{11}}A_{22}^{\dagger}C_{2}B_{22}^{\dagger}+L_{A_{22}}V_{1}+V_{2}R_{B_{11}}+L_{A_{11}}V_{3}R_{B_{22}},
	\end{aligned}
	\end{equation}
	where
	$V_{1}$, $V_{2}$,
	and $V_{3}$ are arbitrary matrices with appropriate sizes over $\mathbb{H}$.
\end{lemma}
\begin{proof}
	$(1)\Leftrightarrow (2)$ It follows from Lemma \ref{lem2.2} that
	\begin{align*}
	&G(A_{22}^{\dagger}C_{2}B_{22}^{\dagger}-A_{11}^{\dagger}C_{1}B_{11}^{\dagger})F=0\\
	\Leftrightarrow &R_{A_{1}}(A_{1}+A_{22}A_{11}^{\dagger}A_{_{11}})A_{22}^{\dagger}C_{2}B_{22}^{\dagger}-A_{11}^{\dagger}C_{1}B_{11}^{\dagger}B_{22}=0\\
	\Leftrightarrow &R_{A_{1}}A_{22}A_{11}^{\dagger}A_{11}A_{22}^{\dagger}C_{2}B_{22}^{\dagger}B_{22}-A_{22}A_{11}^{\dagger}C_{1}B_{11}^{\dagger}B_{22}=0\\
	\Leftrightarrow &R_{A_{1}}A_{22}A_{11}^{\dagger}A_{11}A_{22}^{\dagger}A_{22}A_{22}^{\dagger}C_{2}B_{22}^{\dagger}B_{22}-A_{22}A_{11}^{\dagger}C_{1}B_{11}^{\dagger}B_{22}=0\\
	\Leftrightarrow &R_{A_{1}}(A_{22}-A_{1})A_{22}^{\dagger}A_{22}A_{22}^{\dagger}C_{2}B_{22}^{\dagger}B_{22}-A_{22}A_{11}^{\dagger}C_{1}B_{11}^{\dagger}B_{22}=0\\
	\Leftrightarrow &R_{A_{1}}C_{2}-A_{22}A_{11}^{\dagger}C_{1}B_{11}^{\dagger}B_{22}=0\Leftrightarrow R_{A_{1}}C_{11}=0,
	\end{align*}
	where $G$ and $F$ are given in Lemma \ref{lem2.2}.
	
	$(1)\Rightarrow (3)$ If the system \eqref{eq7} has a solution, then there exists a solution $X_{0}$ such that
	\begin{align*}
	A_{11}X_{0}B_{11}=C_{1}, \ A_{22}X_{0}B_{22}=C_{2}.
	\end{align*}
	It is easy to show that
	\begin{align*}
	R_{A_{ii}}C_{i}=0,\ C_{i}L_{B_{ii}}=0\ (i=1,2).
	\end{align*}
	Thus, $A_{ii}A_{ii}^{\dagger}C_{i}B_{ii}^{\dagger}B_{ii}=C_{i} \ (i=1,2). $ It follows from $R_{B_{11}}B_{22}=0$, $A_{11}L_{A_{22}}=0$ that
	\begin{align*}
	C_{1}B_{11}^{\dagger}B_{22}=A_{11}X_{0}B_{11}B_{11}^{\dagger}B_{22}=A_{11}X_{0}B_{22}=A_{11}A_{22}^{\dagger}A_{22}X_{0}B_{22}=A_{11}A_{22}^{\dagger}C_{2}.
	\end{align*}
	$(3)\Rightarrow (2)$ Since $A_{22}-A_{1}=A_{22}A_{11}^{\dagger}A_{11}$ and $C_{1}B_{11}^{\dagger}B_{22}=A_{11}A_{22}^{\dagger}C_{2}$, we have that
	\begin{align*}
	&R_{A_{1}}C_{11}=R_{A_{1}}C_{2}-R_{A_{1}}A_{22}A_{11}^{\dagger}C_{1}B_{11}^{\dagger}B_{22}= R_{A_{1}} C_{2}-R_{A_{1}}A_{22}A_{11}^{\dagger}A_{11}A_{22}^{\dagger}C_{2}\\
	&= R_{A_{1}} C_{2}-R_{A_{1}}(A_{22}-A_{1})A_{22}^{\dagger}C_{2} =R_{A_{1}}C_{2}-R_{A_{1}}A_{22}A_{22}^{\dagger}C_{2}=0.
	\end{align*}
	$(2)\Leftrightarrow (4)$ It follows from $R_{B_{11}}B_{22}=0$ and $A_{11}L_{A_{22}}=0$ that
	\begin{align*}
	r(B_{22},\ B_{11})=r(B_{11}),\ \begin{pmatrix}
	A_{11}\\
	A_{22}
	\end{pmatrix}=r(A_{22}).
	\end{align*}
	By Lemma \ref{lem2},
	\begin{align*}
	&R_{A_{ii}}C_{i}=0\Leftrightarrow r(R_{A_{ii}}C_{i})=0\Leftrightarrow r(A_{ii},\ C_{i})=r(A_{ii})\ (i=1,2),\\
	& C_{i}L_{B_{ii}}=0\Leftrightarrow r(C_{i}L_{B_{ii}})=0\Leftrightarrow r
	\begin{pmatrix}
	B_{ii} \\
	C_{i} \\
	\end{pmatrix}
	=r(B_{ii})\ (i=1,2),\\
	&R_{A_{1}}C_{11}=0\Leftrightarrow r(R_{A_{1}}C_{11})=0 \Leftrightarrow r(C_{11},\ A_{1})=r(A_{1})\\
	&\Leftrightarrow r
	\begin{pmatrix}
	C_{11} & A_{22}L_{A_{11}} \\
	R_{B_{11}}B_{22} & 0 \\
	\end{pmatrix}
	=r(A_{22}L_{A_{11}})+r(R_{B_{11}}B_{22})\\
	&\Leftrightarrow
	r
	\begin{pmatrix}
	C_{2}-A_{22}A_{11}^{\dagger}C_{1}B_{11}^{\dagger}B_{22} & A_{22} & 0 \\
	B_{22} & 0 & B_{11} \\
	0 & A_{11} & 0 \\
	\end{pmatrix}
	=r
	\begin{pmatrix}
	A_{11} \\
	A_{22} \\
	\end{pmatrix}
	+r(B_{11}, B_{22})\\
	&\Leftrightarrow r
	\begin{pmatrix}
	C_{1} & 0 & A_{11} \\
	0 & -C_{2} & A_{22} \\
	B_{11} & B_{22} & 0 \\
	\end{pmatrix}
	=r
	\begin{pmatrix}
	A_{11} \\
	A_{22} \\
	\end{pmatrix}
	+r(B_{11}, B_{22})=r(A_{22})+r(B_{11}).
	\end{align*}
	
	We now prove that $X_1$ in \eqref{eq8} is the general solution of the system \eqref{eq7}. We prove it in two steps. We show that $X_1$ is a solution of system \eqref{eq7} in Step 1.  In Step 2, if the system \eqref{eq7} is consistent, then the general solution to system \eqref{eq7} can be expressed as \eqref{eq8}.
	
	Step 1. In this step, we show that  $X_1$ is a solution of system \eqref{eq7}. Substituting $X_1$ in \eqref{eq8} into the system \eqref{eq7} yields
	\begin{equation}\label{2}
	\begin{aligned}
	A_{11}X_{1}B_{11}=A_{11}X_{0}B_{11},\ A_{22}X_{1}B_{22}=A_{22}X_{0}B_{22},
	\end{aligned}
	\end{equation}
	where $X_0=A_{11}^{\dagger}C_{1}B_{11}^{\dagger}+L_{A_{11}}A_{22}^{\dagger}C_{2}B_{22}^{\dagger}$. Since $R_{A_{11}}C_{1}=0$ and $C_{1}L_{B_{11}}=0$, we have that
	\begin{align*}
	&A_{11}X_{0}B_{11}=A_{11} A_{11}^{\dagger}C_{1}B_{11}^{\dagger}+L_{A_{11}}A_{1}^{\dagger}A_{22}A_{22}^{\dagger}C_{11}B_{22}^{\dagger}B_{11}\\
	&=A_{11}A_{11}^{\dagger}C_{1}B_{11}^{\dagger}B_{11}+A_{11}L_{A_{11}}A_{1}^{\dagger} C_{11}-R_{A_{22}}C_{11} B_{22}^{\dagger}B_{11}=A_{11}A_{11}^{\dagger}C_{1}B_{11}^{\dagger}B_{11}\\
	&=-R_{A_{11}}C_{1}B_{11}^{\dagger}B_{11}-C_{1}L_{B_{11}}+C_{1}=C_{1}.
	\end{align*}
	
	By
	\begin{align*}
	&R_{B_{11}}B_{22}=0,\ R_{A_{22}}C_{22}=0,\ C_{2}L_{B_{22}}=0\ \operatorname{and}\  C_{1}B_{11}^{\dagger}B_{22}=A_{11}A_{22}^{\dagger}C_{2},
	\end{align*}
	we have that
	\begin{align*}
	&A_{22}X_{0}B_{22}=A_{22}(A_{11}^{\dagger}C_{1}B_{11}^{\dagger}+L_{A_{11}}A_{22}^{\dagger}C_{2}B_{22}^{\dagger})B_{22}\\
	&=A_{22}A_{11}^{\dagger}C_{1}B_{11}^{\dagger}B_{22}+A_{22}A_{22}^{\dagger}C_{2}B_{22}^{\dagger}B_{22}-A_{22}A_{11}^{\dagger}A_{11}A_{22}^{\dagger}C_{2}B_{22}^{\dagger}B_{22}\\
	&=C_{2}+A_{22}A_{11}^{\dagger}C_{1}B_{11}^{\dagger}B_{22}-A_{22}A_{11}^{\dagger}C_{1}B_{11}^{\dagger}B_{22}=C_{2}.
	\end{align*}
	Thus,  $A_{11}X_{1}B_{11}=C_{1}$, $A_{22}X_{1}B_{22}=C_{2}$. $X_1$ is a solution of system \eqref{eq7}.
	
	Step 2. In this step, we show that the general solution to the system \eqref{eq7} can be expressed as \eqref{eq8}. It is sufficient to show that for an arbitrary solution, say, $X_{01}$ of \eqref{eq7}, $X_{01}$ can be expressed in form \eqref{eq8}. Put
	
	$$
	V_{1}=X_{01}B_{22}B_{22}^{\dagger},\ V_{2}=X_{01},\ V_{3}=X_{01}B_{11}B_{11}^{\dagger}.
	$$
	It follows from $B_{22}=B_{11}B_{11}^{\dagger}B_{22}$ and $A_{11}=A_{11}A_{22}^{\dagger}A_{22}$ that
	\begin{align*}
	&X_{1}=A_{11}^{\dagger}C_{1}B_{11}^{\dagger}+L_{A_{11}}A_{22}^{\dagger}C_{2}B_{22}^{\dagger}+L_{A_{22}}V_{1}+V_{2}R_{B_{11}}+L_{A_{11}}V_{3}R_{B_{22}}\\
	&=A_{11}^{\dagger}C_{1}B_{11}^{\dagger}+L_{A_{11}}A_{22}^{\dagger}C_{2}B_{22}^{\dagger}+L_{A_{22}}X_{01}B_{22}B_{22}^{\dagger}+X_{01}R_{B_{11}}+L_{A_{11}}X_{01}B_{11}B_{11}^{\dagger}R_{B_{22}}\\
	&=A_{11}^{\dagger}C_{1}B_{11}^{\dagger}+L_{A_{11}}A_{22}^{\dagger}C_{2}B_{22}^{\dagger}+X_{01}B_{22}B_{22}^{\dagger}-A_{22}^{\dagger}A_{22}X_{01}B_{22}B_{22}^{\dagger}+X_{01}-X_{01}B_{11}B_{11}^{\dagger}\\
	&+X_{01}B_{11}B_{11}^{\dagger}R_{B_{22}}-A_{11}^{\dagger}A_{11}X_{01}B_{11}B_{11}^{\dagger}R_{B_{22}}\\
	&=A_{11}^{\dagger}C_{1}B_{11}^{\dagger}+L_{A_{11}}A_{22}^{\dagger}C_{2}B_{22}^{\dagger}-X_{01}R_{B_{11}}B_{22}B_{22}^{\dagger}+X_{01}-A_{22}^{\dagger}A_{22}X_{01}B_{22}B_{22}^{\dagger}\\
	&-A_{11}^{\dagger}A_{11}X_{01}B_{11}B_{11}^{\dagger}+A_{11}^{\dagger}A_{11}X_{01}B_{11}B_{11}^{\dagger}B_{22}B_{22}^{\dagger}\\
	&=X_{01}+A_{11}^{\dagger}A_{11}X_{01}B_{11}B_{11}^{\dagger}B_{22}B_{22}^{\dagger}-A_{11}^{\dagger}A_{11}A_{22}^{\dagger}A_{22}X_{01}B_{22}B_{22}^{\dagger}\\
	&=X_{01}+A_{11}^{\dagger}A_{11}X_{01}B_{22}B_{22}^{\dagger}-A_{11}^{\dagger}A_{11}X_{01}B_{22}B_{22}^{\dagger}=X_{01}.
	\end{align*}
	Hence, $X_{01}$ can be expressed as \eqref{eq8}. To sum up, \eqref{eq8} is the general solution of the system~\eqref{eq7}.
\end{proof}
Now, we give the fundamental theorem of this paper.
\begin{theorem} \label{Thm1} Let $A_{i},$ $B_{i}$,
	and $B$\ $(i=\overline{1,4})$  be given quaternion matrices with appropriate sizes  over $\mathbb{H}$. Set
	\begin{align}
	&\begin{aligned}\label{21}
	&R_{A_{1}}A_{2}=A_{11},\ R_{A_{1}}A_{3}=A_{22},\ R_{A_{1}}A_{4}=A_{33},\ B_{2}L_{B_{1}}=B_{11},\ B_{22}L_{B_{11}}=N_{1},\\
	&B_{3}L_{B_{1}}=B_{22},\ B_{4}L_{B_{1}}=B_{33},\ R_{A_{11}}A_{22}=M_{1},\ S_{1}=A_{22}L_{M_{1}},\ R_{A_{1}}BL_{B_{1}}=T_{1},\\
	\end{aligned}\\
	&\begin{aligned}\label{22}
	&C=R_{M_{1}}R_{A_{11}},\ C_{1}=CA_{33},\ C_{2}=R_{A_{11}}A_{33},\ C_{3}=R_{A_{22}}A_{33},\ C_{4}=A_{33},\\
	&D=L_{B_{11}}L_{N_{1}},\ D_{1}=B_{33},\ D_{2}=B_{33}L_{B_{22}},\ D_{3}=B_{33}L_{B_{11}},\ D_{4}=B_{33}D,\\
	&E_{1}=CT_{1},\ E_{2}=R_{A_{11}}T_{1}L_{B_{22}},\ E_{3}=R_{A_{22}}T_{1}L_{B_{11}},\ E_{4}=T_{1}D,
	\end{aligned}\\
	&\begin{aligned}\label{23}
	&C_{11}=(L_{C_{2}},\ L_{C_{4}}),\ D_{11}=
	\begin{pmatrix}
	R_{D_{1}} \\
	R_{D_{3}} \\
	\end{pmatrix}
	,\ C_{22}=L_{C_{1}},\ D_{22}=R_{D_{2}},\ C_{33}=L_{C_{3}},\\
	&D_{33}=R_{D_{4}},\ E_{11}=R_{C_{11}}C_{22},\ E_{22}=R_{C_{11}}C_{33},\ E_{33}=D_{22}L_{D_{11}},\ E_{44}=D_{33}L_{D_{11}},\\
	&M=R_{E_{11}}E_{22},\ N=E_{44}L_{E_{33}},\ F=F_{2}-F_{1},\ E=R_{C_{11}}FL_{D_{11}},\ S=E_{22}L_{M},\\
	\end{aligned}\\
	&\begin{aligned}\label{24}
	&F_{11}=C_{2}L_{C_{1}},\ G_{1}=E_{2}-C_{2}C_{1}^{\dagger}E_{1}D_{1}^{\dagger}D_{2},\ F_{22}=C_{4}L_{C_{3}},\  G_{2}=E_{4}-C_{4}C_{3}^{\dagger}E_{3}D_{3}^{\dagger}D_{4},\\
	&F_{1}=C_{1}^{\dagger}E_{1}D_{1}^{\dagger}+L_{C_{1}}C_{2}^{\dagger}E_{2}D_{2}^{\dagger},\ F_{2}=C_{3}^{\dagger}E_{3}D_{3}^{\dagger}+L_{C_{3}}C_{4}^{\dagger}E_{4}D_{4}^{\dagger}.
	\end{aligned}
	\end{align}
	Then,
	the following statements are equivalent:
	
	$\mathrm{(1)}$ Equation \eqref{eq1} is consistent.
	
	$\mathrm{(2)}$
	\begin{equation}\label{new}
	\begin{aligned}
	R_{C_{i}}E_{i}=0,\ E_{i}L_{D_{i}}=0\ (i=\overline{1,4}),\ R_{E_{11}}EL_{E_{44}}=0.
	\end{aligned}
	\end{equation}

	$\mathrm{(3)}$
	\begin{small}
		\begin{align}
		&\begin{aligned}\label{2a}
		&r
		\begin{pmatrix}
		B & A_{2} & A_{3} & A_{4} & A_{1} \\
		B_{1} & 0 & 0 & 0 & 0 \\
		\end{pmatrix}
		=r(B_{1})+r(A_{2},\ A_{3},\ A_{4},\ A_{1}),
		\end{aligned}\\	
		&\begin{aligned}\label{2b}
		&r
		\begin{pmatrix}
		B & A_{2} & A_{4} & A_{1} \\
		B_{3} & 0 & 0 & 0 \\
		B_{1} & 0 & 0 & 0 \\
		\end{pmatrix}
		=r(A_{2},\ A_{4},\ A_{1})+r
		\begin{pmatrix}
		B_{3} \\
		B_{1} \\
		\end{pmatrix},
		\end{aligned}\\		
		&\begin{aligned}\label{2c}
		&r
		\begin{pmatrix}
		B & A_{3} & A_{4} & A_{1} \\
		B_{2} & 0 & 0 & 0 \\
		B_{1} & 0 & 0 & 0 \\
		\end{pmatrix}
		=r(A_{3},\ A_{4},\ A_{1})+r
		\begin{pmatrix}
		B_{2} \\
		B_{1} \\
		\end{pmatrix},
		\end{aligned}\\
		&\begin{aligned}\label{2d}
		&r
		\begin{pmatrix}
		B & A_{4} & A_{1} \\
		B_{2} & 0 & 0 \\
		B_{3} & 0 & 0 \\
		B_{1} & 0 & 0 \\
		\end{pmatrix}
		=r
		\begin{pmatrix}
		B_{2} \\
		B_{3} \\
		B_{1} \\
		\end{pmatrix}
		+r(A_{4},\ A_{1}),
		\end{aligned}\\
		&\begin{aligned}\label{2e}
		&r
		\begin{pmatrix}
		B & A_{2} & A_{3} & A_{1} \\
		B_{4} & 0 & 0 & 0 \\
		B_{1} & 0 & 0 & 0 \\
		\end{pmatrix}
		=r(A_{2},\ A_{3},\ A_{1})+r
		\begin{pmatrix}
		B_{4} \\
		B_{1} \\
		\end{pmatrix},
		\end{aligned}\\
		&\begin{aligned}\label{2f}
		&r
		\begin{pmatrix}
		B & A_{2} & A_{1} \\
		B_{3} & 0 & 0 \\
		B_{4} & 0 & 0 \\
		B_{1} & 0 & 0 \\
		\end{pmatrix}
		=r
		\begin{pmatrix}
		B_{3} \\
		B_{4} \\
		B_{1} \\
		\end{pmatrix}
		+r(A_{2},\ A_{1}),\\
		\end{aligned}
			\end{align}
		\begin{align}
		&\begin{aligned}\label{2g}
		&r
		\begin{pmatrix}
		B & A_{3} & A_{1} \\
		B_{2} & 0 & 0 \\
		B_{4} & 0 & 0 \\
		B_{1} & 0 & 0 \\
		\end{pmatrix}
		=r\begin{pmatrix}
		B_{2} \\
		B_{4} \\
		B_{1} \\
		\end{pmatrix}
		+r(A_{3},\ A_{1}),\\
		\end{aligned}\\
		&\begin{aligned}\label{2h}
		&r
		\begin{pmatrix}
		B & A_{1} \\
		B_{2} & 0 \\
		B_{3} & 0 \\
		B_{4} & 0 \\
		B_{1} & 0 \\
		\end{pmatrix}
		=r
		\begin{pmatrix}
		B_{2} \\
		B_{3} \\
		B_{4} \\
		B_{1} \\
		\end{pmatrix}
		+r(A_{1}),
		\end{aligned}\\	
		&\begin{aligned}\label{2i}
		&r
		\begin{pmatrix}
		B & A_{2} & A_{1} & 0 & 0 & 0 & A_{4} \\
		B_{3} & 0 & 0 & 0 & 0 & 0 & 0 \\
		B_{1} & 0 & 0 & 0 & 0 & 0 & 0 \\
		0 & 0 & 0 & -B & A_{3} & A_{1} & A_{4} \\
		0 & 0 & 0 & B_{2} & 0 & 0 & 0 \\
		0 & 0 & 0 & B_{1} & 0 & 0 & 0 \\
		B_{4} & 0 & 0 & B_{4} & 0 & 0 & 0 \\
		\end{pmatrix}
		&\end{aligned}\\
		&\begin{aligned}\nonumber
		&=r\begin{pmatrix}
		B_{3} & 0 \\
		B_{1} & 0 \\
		0 & B_{2} \\
		0 & B_{1} \\
		B_{4} & B_{4} \\
		\end{pmatrix}+r
		\begin{pmatrix}
		A_{2} & A_{1} & 0 & 0 & A_{4} \\
		0 & 0 & A_{3} & A_{1} & A_{4} \\
		\end{pmatrix}.
		\end{aligned}
		\end{align}	
	\end{small}
\end{theorem}
\begin{proof}
	$(1)\Leftrightarrow (2)$ Equation \eqref{eq1} can be written as
	\begin{equation}\label{eq9}
	\begin{aligned}
	A_{1}X_{1}+X_{2}B_{1}=B-(A_{2}Y_{1}B_{2}+A_{3}Y_{2}B_{3}+A_{4}Y_{3}B_{4}).
	\end{aligned}
	\end{equation}
	
	Clearly, Equation \eqref{eq1} is solvable if and only if Equation \eqref{eq9} has a solution. By Lemma \ref{lem2.3}, Equation \eqref{eq9} is consistent if and only if there exist $Y_i\ (i=\overline{1,3})$ in Equation \eqref{eq9} such that
	\begin{equation}\label{eq10}
	\begin{aligned}
	R_{A_{1}}\left[B-(A_{2}Y_{1}B_{2}+A_{3}Y_{2}B_{3}+A_{4}Y_{3}B_{4})\right]L_{B_{1}}=0,\quad
	\end{aligned}
	\end{equation}
	i.e.,
	\begin{equation}\label{eq11}
	\begin{aligned}
	A_{11}Y_{1}B_{11}+A_{22}Y_{2}B_{22}+A_{33}Y_{3}B_{33}=T_{1},\qquad\qquad\quad\quad
	\end{aligned}
	\end{equation}
	where $A_{ii}$, $B_{ii}(i=\overline{1,3})$,
	and $T_{1}$ are defined by \eqref{21}. In addition, when Equation \eqref{eq9} has a solution, we get the following:
	\begin{align*}
	&\qquad\qquad\quad\quad\quad X_{1}=A_{1}^{\dagger}(B-A_{2}Y_{1}B_{2}-A_{3}Y_{2}B_{3}-A_{4}Y_{3}B_{4})-A_{1}^{\dagger}U_{1}B_{1}+L_{A_{1}}U_{2},\\
	&\qquad\qquad\quad\quad\quad X_{2}=R_{A_{1}}(B-A_{2}Y_{1}B_{2}-A_{3}Y_{2}B_{3}-A_{4}Y_{3}B_{4})B_{1}^{\dagger}+A_{1}A_{1}^{\dagger}U_{1}+U_{3}R_{B_{1}},
	\end{align*}
	where $U_{i}\ (i=\overline{1,3})$ are any matrices with appropriate dimensions over $\mathbb{H}$. Hence, Equation \eqref{eq9} has a solution if and only if there exist $Y_{i}\ (i=\overline{1,3})$ in Equation \eqref{eq9} such that Equation \eqref{eq11} is solvable. According to Equation \eqref{eq11}, we have that
	\begin{equation}\label{eq12}
	\begin{aligned}
	A_{11}Y_{1}B_{11}+A_{22}Y_{2}B_{22}=T_{1}-A_{33}Y_{3}B_{33}.\qquad\qquad\quad\quad
	\end{aligned}
	\end{equation}
	Hence, Equation \eqref{eq11} is consistent if and only if Equation \eqref{eq12} is solvable. It follows from Lemma \ref{lem2.5} that Equation \eqref{eq12} has a solution if and only if there exists $Y_3$ in \mbox{Equation \eqref{eq12}} such that
	\begin{equation}\label{eq13}
	\begin{aligned}
	&\qquad\qquad\quad\quad R_{M_{1}}R_{A_{11}}(A_{33}Y_{3}B_{33}-T_{1})=0,\ R_{A_{11}}(T_{1}-A_{33}Y_{3}B_{33})L_{B_{22}}=0,\\
	&\qquad\qquad\quad\quad R_{A_{22}}(T_{1}-A_{33}Y_{3}B_{33})L_{B_{11}}=0,\ (T_{1}-A_{33}Y_{3}B_{33})L_{B_{11}}L_{N_{1}}=0,
	\end{aligned}
	\end{equation}
	i.e.,
	\begin{equation}\label{eq16}
	\begin{aligned}
	C_{1}Y_{3}D_{1}=E_{1},\ C_{2}Y_{3}D_{2}=E_{2},\ C_{3}Y_{3}D_{3}=E_{3},\ C_{4}Y_{3}D_{4}=E_{4},
	\end{aligned}
	\end{equation}
	where $C_{i}$, $D_{i}$, $E_{i}\ (i=\overline{1,4})$ are defined by \eqref{22}. When Equation \eqref{eq12} is solvable, we have that
	\begin{align*}
	&\qquad\qquad \ \ Y_{1}=A_{11}^{\dagger}TB_{11}^{\dagger}-A_{11}^{\dagger}A_{22}M_{1}^{\dagger}TB_{11}^{\dagger}-A_{11}^{\dagger}S_{1}A_{22}^{\dagger}TN_{1}^{\dagger}B_{22}B_{11}^{\dagger}\\
	&\qquad\qquad \ \ -A_{11}^{\dagger}S_{1}U_{4}R_{N_{1}}B_{22}B_{11}^{\dagger}+L_{A_{11}}U_{5}+U_{6}R_{B_{11}},\\
	&\qquad\qquad \ \ Y_{2}=M_{1}^{\dagger}TB_{22}^{\dagger}+S_{1}^{\dagger}S_{1}A_{22}^{\dagger}TN_{1}^{\dagger}+L_{M_{1}}L_{S_{1}}U_{7}+U_{8}R_{B_{22}}+L_{M_{1}}U_{4}R_{N_{1}},
	\end{align*}
	where $A_{ii},\ B_{ii}\ (i=\overline{1,3})$,\ $M_{1},\ N_{1},\ S_{1},\ T_{1}$ are defined by \eqref{21}, $T=T_{1}-A_{33}Y_{3}B_{33}$ and $U_{j}\ (j=\overline{4,8})$ are any matrices with the appropriate dimensions  over $\mathbb{H}$.
	
	It is easy to infer that
	\begin{equation}\label{eq17}
	\begin{aligned}
	C_{1}L_{C_{2}}=0,\ R_{D_{1}}D_{2}=0,\ C_{3}L_{C_{4}}=0,\ R_{D_{3}}D_{4}=0.\qquad
	\end{aligned}
	\end{equation}
	Thus, according to Lemma \ref{lem3.1}, we have that the system \eqref{eq16} is consistent if and only if
	\begin{equation}\label{equ1}
	\begin{aligned}
	\qquad \quad \ R_{C_{i}}E_{i}=0,\ E_{i}L_{D_{i}}=0\ (i=1,2,3,4),\ R_{F_{11}}G_{1}=0,\ R_{F_{22}}G_{2}=0.
	\end{aligned}
	\end{equation}
	In this case, the general solution to system \eqref{eq16} can be expressed as
	\begin{equation}\label{eq18}
	\begin{aligned}
	Y_{3}=F_{1}+L_{C_{2}}V_{1}+V_{2}R_{D_{1}}+L_{C_{1}}V_{3}R_{D_{2}},\qquad \quad \ \qquad \quad \
	\end{aligned}
	\end{equation}
	\begin{equation}\label{eq19}
	\begin{aligned}
	Y_{3}=F_{2}-L_{C_{4}}W_{1}-W_{2}R_{D_{3}}-L_{C_{3}}W_{3}R_{D_{4}},\qquad \quad  \qquad \quad
	\end{aligned}
	\end{equation}
	where $F_{1}$, $F_{2}$ are defined by \eqref{24} and $V_{i},$ $W_{i}\ (i=\overline{1,3})$ are any matrices with the appropriate dimensions  over $\mathbb{H}$.
	Thus, system \eqref{eq16} has a solution if and only if \eqref{equ1} holds and there exist $V_{i}, W_{i}\ (i=\overline{1,3})$ such that \eqref{eq18} equals
	to \eqref{eq19}, namely
	\begin{align*}
	\qquad\quad\quad\quad	(L_{C_{2}},\ L_{C_{4}})
	\begin{pmatrix}
	V_{1} \\
	W_{1} \\
	\end{pmatrix}
	+(V_{2},\ W_{2})
	\begin{pmatrix}
	R_{D_{1}} \\
	R_{D_{3}} \\
	\end{pmatrix}+L_{C_{1}}V_{3}R_{D_{2}}+L_{C_{3}}W_{3}R_{D_{4}}=F,
	\end{align*}
	i.e.,
	\begin{equation}\label{eq20}
	\begin{aligned}
	C_{11}
	\begin{pmatrix}
	V_{1} \\
	W_{1} \\
	\end{pmatrix}
	+(V_{2},\ W_{2})D_{11}
	+C_{22}V_{3}D_{22}+C_{33}W_{3}D_{33}=F,
	\end{aligned}
	\end{equation}
where $F$, $C_{ii}$ and $D_{ii}\ (i=\overline{1,3})$ are defined by \eqref{23}. It follows from Lemma \ref{lem2.5} that \mbox{Equation~\eqref{eq20}} has a solution if and only if
	\begin{equation}\label{eq22}
	\begin{aligned}
	\qquad\quad\quad\qquad	R_{M}R_{E_{11}}E=0,\ EL_{E_{33}}L_{N}=0,\ R_{E_{11}}EL_{E_{44}}=0,\ R_{E_{22}}EL_{E_{33}}=0.
	\end{aligned}
	\end{equation}
	In this case, the general solution to  Equation \eqref{eq20} can be expressed as           
		\begin{align*}
		&\qquad \qquad \qquad \quad \quad V_{1}=(I_{m},\ 0)\left[C_{11}^{\dagger}(F-C_{22}V_{3}D_{22}-C_{33}W_{3}D_{33})-C_{11}^{\dagger}U_{11}D_{11}+L_{C_{11}}U_{12}\right],\\
		&\qquad \qquad \qquad \quad \quad W_{1}=(0,\ I_{m})\left[C_{11}^{\dagger}(F-C_{22}V_{3}D_{22}-C_{33}W_{3}D_{33})-C_{11}^{\dagger}U_{11}D_{11}+L_{C_{11}}U_{12}\right],\\
		&\qquad \qquad \qquad \quad \quad W_{2}=\left[R_{C_{11}}(F-C_{22}V_{3}D_{22}-C_{33}W_{3}D_{33})D_{11}^{\dagger}+C_{11}C_{11}^{\dagger}U_{11}+U_{21}R_{D_{11}}\right]
		\begin{pmatrix}
		0 \\
		I_{n} \\
		\end{pmatrix},\\
		&\qquad \qquad \qquad \quad \quad V_{2}=\left[R_{C_{11}}(F-C_{22}V_{3}D_{22}-C_{33}W_{3}D_{33})D_{11}^{\dagger}+C_{11 }(C_{11})^{\dagger}U_{11}+U_{21}R_{D_{11}}\right]
		\begin{pmatrix}
		I_{n} \\
		0 \\
		\end{pmatrix},\\
		&\qquad \qquad \qquad \quad \quad V_{3}=E_{11}^{\dagger}FE_{33}^{\dagger}-E_{11}^{\dagger}E_{22}M^{\dagger}FE_{33}^{\dagger}-E_{11}^{\dagger}SE_{22}^{\dagger}FN^{\dagger}E_{44}E_{33}^{\dagger}\\
		&\qquad \qquad \qquad \quad \quad -E_{11}^{\dagger}SU_{31}R_{N}E_{44}E_{33}^{\dagger}+L_{E_{11}}U_{32}+U_{33}R_{E_{33}},\\
		&\qquad \qquad \qquad \quad \quad W_{3}=M^{\dagger}FE_{44}^{\dagger}+S^{\dagger}SE_{22}^{\dagger}FN^{\dagger}+L_{M}L_{S}U_{41}+L_{M}U_{31}R_{N}-U_{42}R_{E_{44}},
		\end{align*}   
	where $U_{11}, U_{12}$, $U_{21}$, $U_{31}$, $U_{32}$, $U_{33}$, $U_{41}$,
	and $U_{42}$ are any matrices with the suitable dimensions over $\mathbb{H}$. $M$, $E$, $N$, $S$, $C_{11}$, $D_{11}$,
	and $E_{ii}\ (i=\overline{1,4})$ are defined by \eqref{23}, $m$ is the column number of $A_4$ and $n$ is the row number of $B_4$. 
	We summarize up that \eqref{eq11} has a solution if and only if \eqref{equ1} and \eqref{eq22} hold. Hence, Equation \eqref{eq1} is solvable if and only if \eqref{equ1} and~\eqref{eq22}~hold.
	
	In fact,  $R_{C_2}E_2=0,\ E_1L_{D_1}=0$ $\Rightarrow$ $R_{F_{11}}G_1=0$;  $R_{C_4}E_4=0,\ E_3L_{D_3}=0$ $\Rightarrow$ $R_{F_{22}}G_2=0$; $R_{C_3}E_3=0,\ E_1L_{D_1}=0$ $\Rightarrow$ $R_MR_{E_{11}}E=0$;  $R_{C_4}E_4=0,\ E_1L_{D_1}=0$ $\Rightarrow$ $EL_{E_{33}}L_{N}=0$; $R_{C_4}E_4=0,\ E_2L_{D_2}=0$ $\Rightarrow$ $R_{E_{22}}EL_{E_{33}}=0$. 
	The specific proof is as follows.
	
	Firstly, we prove that  $R_{C_2}E_2=0,\ E_1L_{D_1}=0$ $\Rightarrow$ $R_{F_{11}}G_1=0$;  $R_{C_4}E_4=0,\ E_3L_{D_3}=0$ $\Rightarrow$ $R_{F_{22}}G_2=0$. It follows from Lemma \ref{lem2} and elementary transformations that
	\begin{small}
		\begin{align}
		&\begin{aligned}\label{eq35}
		&R_{C_{1}}E_{1}=0\Leftrightarrow r(E_{1},\ C_{1})=r(C_{1})=r(CT_{1},\ CA_{33})=r(CA_{33})\ \Leftrightarrow \\
		& r(T_{1},\ A_{33},\ A_{11},\ A_{22})=r(A_{33},\ A_{11},\ A_{22}),
		\end{aligned}\\
		&\begin{aligned}\label{eq36}
		&R_{C_{2}}E_{2}=0\Leftrightarrow r(E_{2},\ C_{2})=r(C_{2}) \Leftrightarrow
		r\begin{pmatrix}
		T_{1} & A_{33} & A_{11} \\
		B_{22} & 0 & 0 \\
		\end{pmatrix}
		=r(A_{33} ,\ A_{11})+r(B_{22}),
		\end{aligned}\\
		&\begin{aligned}\label{eq37}
		&R_{C_{3}}E_{3}=0\Leftrightarrow r(E_{3},\ C_{3})=r(C_{3})\Leftrightarrow
		r\begin{pmatrix}
		T_{1} & A_{33} & A_{22} \\
		B_{11} & 0 & 0 \\
		\end{pmatrix}
		=r(A_{33} ,\ A_{22})+r(B_{11}),
		\end{aligned}\\
		&\begin{aligned}\label{eq38}
		&R_{C_{4}}E_{4}=0\Leftrightarrow r(E_{4},\ C_{4})=r(C_{4})\Leftrightarrow
		r\begin{pmatrix}
		T_{1} & A_{33} \\
		B_{11} & 0 \\
		B_{22} & 0 \\
		\end{pmatrix}
		=r(A_{33})+r
		\begin{pmatrix}
		B_{11} \\
		B_{22} \\
		\end{pmatrix}
		,
		\end{aligned}\\
		&\begin{aligned}\label{eq39}
		&E_{1}L_{D_{1}}=0\Leftrightarrow r
		\begin{pmatrix}
		E_{1} \\
		D_{1} \\
		\end{pmatrix}
		\Leftrightarrow r
		\begin{pmatrix}
		T_{1} & A_{11} & A_{22} \\
		B_{33} & 0 & 0 \\
		\end{pmatrix}
		=r(A_{11} ,\ A_{22})+r(B_{33}),
		\end{aligned}\\
		&\begin{aligned}\label{eq3.49}
		&E_{2}L_{D_{2}}=0\Leftrightarrow r
		\begin{pmatrix}
		E_{2} \\
		D_{2} \\
		\end{pmatrix}
		=r(D_{2})\Leftrightarrow r
		\begin{pmatrix}
		T_{1} & A_{11} \\
		B_{33} & 0 \\
		B_{22} & 0 \\
		\end{pmatrix}
		=r
		\begin{pmatrix}
		B_{33} \\
		B_{22} \\
		\end{pmatrix}
		+r(A_{11}),
		\end{aligned}\\
		&\begin{aligned}\label{eq40}
		& E_{3}L_{D_{3}}=0\Leftrightarrow  r
		\begin{pmatrix}
		E_{3} \\
		D_{3} \\
		\end{pmatrix}
		=r(D_{3})\Leftrightarrow r
		\begin{pmatrix}
		T_{1} & A_{22} \\
		B_{33} & 0 \\
		B_{11} & 0 \\
		\end{pmatrix}
		=r
		\begin{pmatrix}
		B_{33} \\
		B_{11} \\
		\end{pmatrix}
		+r(A_{22}),\quad\quad\quad
		\end{aligned}.
		\end{align}
		\begin{align}
		&\begin{aligned}\label{eq41}
		&E_{4}L_{D_{4}}=0\Leftrightarrow r
		\begin{pmatrix}
		E_{4} \\
		D_{4} \\
		\end{pmatrix}
		=r(D_{4})\Leftrightarrow r
		\begin{pmatrix}
		T_{1} \\
		B_{33} \\
		B_{11} \\
		B_{22} \\
		\end{pmatrix}
		=r
		\begin{pmatrix}
		B_{33} \\
		B_{11} \\
		B_{22} \\
		\end{pmatrix}
		.
		\end{aligned}
		\end{align}	
		
		It follows from Lemma \ref{lem3.1} and \eqref{eq17} that $R_{F_{11}}G_{1}=0$ and $R_{F_{22}}G_{2}=0$ are equivalent to
		\begin{equation}\label{eqc1}
		\begin{aligned}
		r
		\begin{pmatrix}
		E_{1} & 0 & C_{1} \\
		0 & -E_{2} & C_{2} \\
		D_{1} & D_{2} & 0 \\ 
		\end{pmatrix}
		=r
		\begin{pmatrix}
		C_{1}\\
		C_{2} \\
		\end{pmatrix}
		+r(D_{1},\ D_{2}),\quad\quad\quad\quad\quad\quad\quad\quad\quad\quad\quad\quad\quad\quad\quad\quad\quad\quad
		\end{aligned}
		\end{equation}
		\begin{equation}\label{eqc2}
		\begin{aligned}
		r
		\begin{pmatrix}
		E_{3} & 0 & C_{3} \\
		0 & -E_{4} & C_{4} \\
		D_{3} & D_{4} & 0 \\
		\end{pmatrix}
		=r
		\begin{pmatrix}
		C_{3}\\
		C_{4} \\
		\end{pmatrix}
		+r(D_{3},\ D_{4}).\quad\quad\quad\quad\quad\quad\quad\quad\quad\quad\quad\quad\quad\quad\quad\quad\quad\quad
		\end{aligned}
		\end{equation}	
		
		According to Lemma \ref{lem2}, we have that
		\begin{equation}\label{eqc3}
		\begin{aligned}
		&\eqref{eqc1}\\
		&\Leftrightarrow r
		\begin{pmatrix}
		T_{1} & 0 & 0 & A_{11} & A_{22} & 0 \\
		0 & -T_{1} & A_{33} & 0 & 0 & A_{11} \\
		B_{33} & 0 & 0 & 0 & 0 & 0 \\
		0 & B_{22} & 0 & 0 & 0 & 0 \\
		\end{pmatrix}=r
		\begin{pmatrix}
		0 & A_{11} & A_{22} & 0 \\
		A_{33} & 0 & 0 & A_{11} \\
		\end{pmatrix}
		+r
		\begin{pmatrix}
		B_{33} & 0 \\
		0 & B_{22} \\
		\end{pmatrix}
		\\
		&\Leftrightarrow r
		\begin{pmatrix}
		T_{1} & A_{11} & A_{22} & 0 & 0 & 0 \\
		B_{33} & 0 & 0 & 0 & 0 & 0 \\
		0 & 0 & 0 & T_{1} & A_{33} & A_{11} \\
		0 & 0 & 0 & B_{22} & 0 & 0 \\
		\end{pmatrix}=r
		\begin{pmatrix}
		A_{11} & A_{22} & 0 & 0 \\
		0 & 0 & A_{33} & A_{11} \\
		\end{pmatrix}
		+r
		\begin{pmatrix}
		B_{33} & 0 \\
		0 & B_{22} \\
		\end{pmatrix}
		.
		\end{aligned}
		\end{equation}
	\end{small}	
	Thus, it follows from \eqref{eqc3} that \eqref{eqc1} holds when \eqref{eq36} and \eqref{eq39} hold. Similarly, if \eqref{eq38} and \eqref{eq40} hold, then \eqref{eqc2} holds.  
	
	Secondly,  we prove that  $R_{C_3}E_3=0,\ E_1L_{D_1}=0$ $\Rightarrow$ $R_MR_{E_{11}}E=0$;  $R_{C_4}E_4=0,\ E_1L_{D_1}=0$ $\Rightarrow$ $EL_{E_{33}}L_{N}=0$; $R_{C_4}E_4=0,\ E_2L_{D_2}=0$ $\Rightarrow$ $R_{E_{22}}EL_{E_{33}}=0$.
	According to Lemma \ref{lem2.5} and~\eqref{eq17}, we have that \eqref{eq22} are equivalent to
	\begin{equation}\label{eq23}
	\begin{aligned}
	r
	\begin{pmatrix}
	F & L_{C_{1}} & L_{C_{3}} \\
	R_{D_{1}} & 0 & 0 \\
	R_{D_{3}} & 0 & 0 \\
	\end{pmatrix}
	=r(L_{C_{1}},\ L_{C_{3}})+r
	\begin{pmatrix}
	R_{D_{1}} \\
	R_{D_{3}} \\
	\end{pmatrix}
	,
	\end{aligned}
	\end{equation}
	\begin{equation}\label{eq24}
	\begin{aligned}
	r
	\begin{pmatrix}
	F & L_{C_{2}} & L_{C_{4}} \\
	R_{D_{2}} & 0 & 0 \\
	R_{D_{4}} & 0 & 0 \\
	\end{pmatrix}
	=r(L_{C_{2}},\ L_{C_{4}})+r
	\begin{pmatrix}
	R_{D_{2}} \\
	R_{D_{4}} \\
	\end{pmatrix}
	,
	\end{aligned}
	\end{equation}
	\begin{equation}\label{eq25}
	\begin{aligned}
	r
	\begin{pmatrix}
	F & L_{C_{1}} & L_{C_{4}} \\
	R_{D_{1}} & 0 & 0 \\
	R_{D_{4}} & 0 & 0 \\
	\end{pmatrix}
	=r(L_{C_{1}},\ L_{C_{4}})+r
	\begin{pmatrix}
	R_{D_{1}} \\
	R_{D_{4}} \\
	\end{pmatrix}
	,
	\end{aligned}
	\end{equation}
	\begin{equation}\label{eq26}
	\begin{aligned}
	r
	\begin{pmatrix}
	F & L_{C_{2}} & L_{C_{3}} \\
	R_{D_{2}} & 0 & 0 \\
	R_{D_{3}} & 0 & 0 \\
	\end{pmatrix}
	=r(L_{C_{2}},\ L_{C_{3}})+r
	\begin{pmatrix}
	R_{D_{2}} \\
	R_{D_{3}} 
	\end{pmatrix}
	,
	\end{aligned}
	\end{equation}
	respectively. By Lemma \ref{lem2}, we have that
	\begin{equation}\label{eq28}
	\begin{aligned}
	&\eqref{eq23}\\
	&\Leftrightarrow r
	\begin{pmatrix}
	F & I & I & 0 & 0 \\
	I & 0 & 0 & D_{1} & 0 \\
	I & 0 & 0 & 0 & D_{3} \\
	0 & C_{1} & 0 & 0 & 0 \\
	0 & 0 & C_{3} & 0 & 0 \\
	\end{pmatrix}=r
	\begin{pmatrix}
	I & D_{1} & 0 \\
	I & 0 & D_{3} \\
	\end{pmatrix}
	+r
	\begin{pmatrix}
	I & I \\
	C_{1} & 0 \\
	0 & C_{3} \\
	\end{pmatrix}\\	
	&\Leftrightarrow r
	\begin{pmatrix}
	E_{1} & 0 & C_{1} \\
	0 & -E_{3} & C_{3} \\
	D_{1} & D_{3} & 0 \\
	\end{pmatrix}
	=r
	\begin{pmatrix}
	C_{1} \\
	C_{3} \\
	\end{pmatrix}
	+r(D_{1},\ D_{3}).
	\end{aligned}
	\end{equation}
	Similarly, we can show that \eqref{eq24}--\eqref{eq26} are equivalent to
	\begin{equation}\label{eq29}
	\begin{aligned}
	r
	\begin{pmatrix}
	E_{1} & 0 & C_{1} \\
	0 & -E_{4} & C_{4} \\
	D_{1} & D_{4} & 0 \\
	\end{pmatrix}
	=r
	\begin{pmatrix}
	C_{1} \\
	C_{4} \\
	\end{pmatrix}
	+r(D_{1},\ D_{4}),
	\end{aligned}
	\end{equation}
	\begin{equation}\label{eq30}
	\begin{aligned}
	r
	\begin{pmatrix}
	E_{2} & 0 & C_{2} \\
	0 & -E_{3} & C_{3} \\
	D_{2} & D_{3} & 0 \\
	\end{pmatrix}
	=r
	\begin{pmatrix}
	C_{2} \\
	C_{3} \\
	\end{pmatrix}
	+r(D_{2},\ D_{3}),
	\end{aligned}
	\end{equation}
	\begin{equation}\label{eq31}
	\begin{aligned}
	r
	\begin{pmatrix}
	E_{2} & 0 & C_{2} \\
	0 & -E_{4} & C_{4} \\
	D_{2} & D_{4} & 0 \\
	\end{pmatrix}
	=r
	\begin{pmatrix}
	C_{2} \\
	C_{4} \\
	\end{pmatrix}
	+r(D_{2},\ D_{4}).
	\end{aligned}
	\end{equation}
	
	Substituting $C_i, D_i$, and $E_i\ (i=1,3)$ in \eqref{22} into the rank equality \eqref{eq28} and by Lemma \ref{lem2}, we have that
	\begin{small}
		\begin{equation}\label{eq42}
		\begin{aligned}
		&\eqref{eq28}
		\\
		&\Leftrightarrow r
		\begin{pmatrix}
		T_{1} & 0 & 0 & A_{11} & A_{22} & 0 \\
		0 & -T_{1} & A_{33} & 0 & 0 & A_{22} \\
		B_{33} & 0 & 0 & 0 & 0 & 0 \\
		0 & B_{11} & 0 & 0 & 0 & 0 \\
		\end{pmatrix}=r
		\begin{pmatrix}
		0 & A_{11} & A_{22} & 0 \\
		A_{33} & 0 & 0 & A_{22} \\
		\end{pmatrix}
		+r
		\begin{pmatrix}
		B_{33} & 0 \\
		0 & B_{11} \\
		\end{pmatrix}\\
		&\Leftrightarrow r
		\begin{pmatrix}
		T_{1} & A_{11} & A_{22} & 0 & 0 & 0 \\
		B_{33} & 0 & 0 & 0 & 0 & 0 \\
		0 & 0 & 0 & T_{1} & A_{33} & A_{22} \\
		0 & 0 & 0 & B_{11} & 0 & 0 \\
		\end{pmatrix}=r
		\begin{pmatrix}
		A_{11} & A_{22} & 0 & 0 \\
		0 & 0 & A_{33} & A_{22} \\
		\end{pmatrix}
		+r
		\begin{pmatrix}
		B_{33} & 0 \\
		0 & B_{11} \\
		\end{pmatrix}
		.
		\end{aligned}
		\end{equation}
	\end{small}
Hence, it follows from \eqref{eq37} and \eqref{eq39} that \eqref{eq42} holds. Similarly, we can prove that when~\eqref{eq38},~\eqref{eq39} hold and \eqref{eq38}, \eqref{eq3.49} hold, we can get that \eqref{eq29} and \eqref{eq31} hold, respectively.
Thus, Equation \eqref{eq11} has a solution if and only if \eqref{new} holds. That is to say, Equation~\eqref{eq1} has a solution if and only if~\eqref{new} holds.

	$(2)\Leftrightarrow (3)$ We prove the equivalence in two parts. In the first part, we want to show that \eqref{eq35} to \eqref{eq41} are equivalent to \eqref{2a} to \eqref{2h}, respectively. In the second part, we want to show that \eqref{eq30} is equivalent to \eqref{2i}.
	
	Part 1. We want to show that \eqref{eq35} to \eqref{eq41} are equivalent to \eqref{2a} to \eqref{2h}, respectively. It follows from Lemma \ref{lem2} and elementary operations to  \eqref{eq35} that
	\begin{align*}
	&\eqref{eq35}\Leftrightarrow r(R_{A_{1}}BL_{B_{11}},\ R_{A_{1}}A_{4},\ R_{A_{1}}A_{2},\ R_{A_{1}}A_{3})=r(R_{A_{1}}A_{4},\ R_{A_{1}}A_{2},\ R_{A_{1}}A_{3})\\
	&\Leftrightarrow r
	\begin{pmatrix}
	B & A_{4} & A_{2} & A_{3} & A_{1} \\
	B_{1} & 0 & 0 & 0 & 0 
	\end{pmatrix}
	=r(A_{4} ,\ A_{2} ,\ A_{3} ,\ A_{1})+r(B_{1})\Leftrightarrow (17). 
	\end{align*}
	
	Similarly, we can show that \eqref{eq36} to \eqref{eq38} are equivalent to \eqref{2b} to \eqref{2d}, respectively. 
	Now,
	we turn to prove that \eqref{eq39} is equivalent to \eqref{2c}. It follows from the Lemma \ref{lem2} and elementary transformations that
	\begin{align*}
	&\eqref{eq39}\Leftrightarrow r
	\begin{pmatrix}
	R_{A_{1}}BL_{B_{1}} & R_{A_{1}}A_{2} & R_{A_{1}}A_{3} \\
	B_{4}L_{B_{1}} & 0 & 0 \\
	\end{pmatrix}
	=r( R_{A_{1}}A_{2} ,\ R_{A_{1}}A_{3})+r(B_{4}L_{B_{1}})\\
	&\Leftrightarrow r
	\begin{pmatrix}
	B & A_{2} & A_{3} & A_{1} \\
	B_{4} & 0 & 0 & 0 \\
	B_{1} & 0 & 0 & 0 \\
	\end{pmatrix}
	=r(A_{2} ,\ A_{3} ,\ A_{1})+r\begin{pmatrix}
	B_{4} \\
	B_{1} \\
	\end{pmatrix}\Leftrightarrow (21).       
	\end{align*}
	
	Similarly, we can show that \eqref{eq3.49} to \eqref{eq41} are equivalent to \eqref{2f} to \eqref{2h}. Hence, \eqref{eq35} to \eqref{eq41} are equivalent to \eqref{2a} to \eqref{2h}, respectively.
	
	Part 2. We want to show that $\eqref{eq30}\Leftrightarrow \eqref{2i}$. It follows from Lemma \ref{lem2} and elementary operations to \eqref{eq30} that
	\begin{small}
		\begin{align*}
		&\eqref{eq30}\Leftrightarrow \\
		&\Leftrightarrow r\begin{pmatrix}
		R_{A_{11}}T_{1}L_{B_{22}} & 0 & R_{A_{11}}A_{33} \\
		0 & -R_{_{A_{22}}}T_{1}L_{B_{11}} & R_{A_{22}}A_{33} \\
		B_{33}L_{B_{22}} & B_{33}L_{B_{11}} & 0
		\end{pmatrix}=r
		\begin{pmatrix}
		R_{A_{11}}A_{33} \\
		R_{A_{22}}A_{33} \\
		\end{pmatrix}
		+r(B_{33}L_{B_{22}},\ B_{33}L_{B_{11}})\\
		&\Leftrightarrow r
		\begin{pmatrix}
		T_{1} & 0 & A_{11} & 0 & A_{33} \\
		0 & -T_{1} & 0 & A_{22} & A_{33} \\
		B_{22} & 0 & 0 & 0 & 0 \\
		0 & B_{11} & 0 & 0 & 0 \\
		B_{33} & B_{33} & 0 & 0 & 0 \\
		\end{pmatrix}=r
		\begin{pmatrix}
		B_{22} & 0 \\
		0 & B_{11} \\
		B_{33} & B_{33} \\
		\end{pmatrix}+r
		\begin{pmatrix}
		A_{11} & 0 & A_{33} \\
		0 & A_{22} & A_{33} \\
		\end{pmatrix}\\
		&\Leftrightarrow r	\begin{pmatrix}
		B_{3} & 0 \\
		0 & B_{2} \\
		B_{4} & B_{4} \\
		B_{1} & 0 \\
		0 & B_{1} \\
		\end{pmatrix}
		+r
		\begin{pmatrix}
		A_{2} & 0 & A_{4} & A_{1} & 0 \\
		0 & A_{3} & A_{4} & 0 & A_{1} \\
		\end{pmatrix}\\
		& =r\begin{pmatrix}
		B & 0 & A_{2} & 0 & A_{4} & A_{1} & 0 \\
		0 & -B & 0 & A_{3} & A_{4} & 0 & A_{1} \\
		B_{3} & 0 & 0 & 0 & 0 & 0 & 0 \\
		0 & B_{2} & 0 & 0 & 0 & 0 & 0 \\
		B_{4} & B_{4} & 0 & 0 & 0 & 0 & 0 \\
		B_{1} & 0 & 0 & 0 & 0 & 0 & 0 \\
		0 & B_{1} & 0 & 0 & 0 & 0 & 0 \\
		\end{pmatrix}\Leftrightarrow \eqref{2i}.
		\end{align*}
	\end{small}
	Hence, \eqref{eq35} to \eqref{eq41} and \eqref{eq30} are equivalent to \eqref{2a} to \eqref{2i}, respectively.
\end{proof}
Next, we give the formula of general solution to matrix Equation \eqref{eq1} by using 
Moore--Penrose. According to Theorem \ref{Thm1}, we get the following theorem:
\begin{theorem}\label{Thm2}
	Let matrix Equation \eqref{eq1} be solvable. 
	Then,
	the general solution to matrix Equation \eqref{eq1} can be expressed as
	\begin{align*}
	&X_{1}=A_{1}^{\dagger}(B-A_{2}Y_{1}B_{2}-A_{3}Y_{2}B_{3}-A_{4}Y_{3}B_{4})-A_{1}^{\dagger}U_{1}B_{1}+L_{A_{1}}U_{2},\\
	&X_{2}=R_{A_{1}}(B-A_{2}Y_{1}B_{2}-A_{3}Y_{2}B_{3}-A_{4}Y_{3}B_{4})B_{1}^{\dagger}+A_{1}A_{1}^{\dagger}U_{1}+U_{3}R_{B_{1}},\\
	&Y_{1}=A_{11}^{\dagger}TB_{11}^{\dagger}-A_{11}^{\dagger}A_{22}M_{1}^{\dagger}TB_{11}^{\dagger}-A_{11}^{\dagger}S_{1}A_{22}^{\dagger}TN_{1}^{\dagger}B_{22}B_{11}^{\dagger}\\
	&-A_{11}^{\dagger}S_{1}U_{4}R_{N_{1}}B_{22}B_{11}^{\dagger}+L_{A_{11}}U_{5}+U_{6}R_{B_{11}},\\
	&Y_{2}=M_{1}^{\dagger}TB_{22}^{\dagger}+S_{1}^{\dagger}S_{1}A_{22}^{\dagger}TN_{1}^{\dagger}+L_{M_{1}}L_{S_{1}}U_{7}+U_{8}R_{B_{22}}+L_{M_{1}}U_{4}R_{N_{1}},\\
	&Y_{3}=F_{1}+L_{C_{2}}V_{1}+V_{2}R_{D_{1}}+L_{C_{1}}V_{3}R_{D_{2}},\ or \ Y_{3}=F_{2}-L_{C_{4}}W_{1}-W_{2}R_{D_{3}}-L_{C_{3}}W_{3}R_{D_{4}},
	\end{align*}
	where $T=T_{1}-A_{33}Y_{3}B_{33}$, $U_{i}(i=\overline{1,8})$ are arbitrary matrices with appropriate sizes over $\mathbb{H}$,         
		\begin{align*}
		&V_{1}=(I_{m},\ 0)\left[C_{11}^{\dagger}(F-C_{22}V_{3}D_{22}-C_{33}W_{3}D_{33})-C_{11}^{\dagger}U_{11}D_{11}+L_{C_{11}}U_{12}\right],\quad \quad\quad \quad\quad\quad \quad \quad \quad \\
		&W_{1}=(0,\ I_{m})\left[C_{11}^{\dagger}(F-C_{22}V_{3}D_{22}-C_{33}W_{3}D_{33})-C_{11}^{\dagger}U_{11}D_{11}+L_{C_{11}}U_{12}\right],\\
		&W_{2}=\left[R_{C_{11}}(F-C_{22}V_{3}D_{22}-C_{33}W_{3}D_{33})D_{11}^{\dagger}+C_{11}C_{11}^{\dagger}U_{11}+U_{21}R_{D_{11}}\right]
		\begin{pmatrix}
		0 \\
		I_{n} \\
		\end{pmatrix}
		,\\
		&V_{2}=\left[R_{C_{11}}(F-C_{22}V_{3}D_{22}-C_{33}W_{3}D_{33})D_{11}^{\dagger}+C_{11}C_{11}^{\dagger}U_{11}+U_{21}R_{D_{11}}\right]
		\begin{pmatrix}
		I_{n} \\
		0 \\                                                       \end{pmatrix}                                                   ,\\
		&V_{3}=E_{11}^{\dagger}FE_{33}^{\dagger}-E_{11}^{\dagger}E_{22}M^{\dagger}FE_{33}^{\dagger}-E_{11}^{\dagger}SE_{22}^{\dagger}FN^{\dagger}E_{44}E_{33}^{\dagger}-E_{11}^{\dagger}SU_{31}R_{N}E_{44}E_{33}^{\dagger}+L_{E_{11}}U_{32}+U_{33}R_{E_{33}},\\
		&W_{3}=M^{\dagger}FE_{44}^{\dagger}+S^{\dagger}SE_{22}^{\dagger}FN^{\dagger}+L_{M}L_{S}U_{41}+L_{M}U_{31}R_{N}-U_{42}R_{E_{44}},
		\end{align*}        
	$U_{11}, U_{12}$, $U_{21}$, $U_{31}$, $U_{32}$, $U_{33}$, $U_{41}$,
	and $U_{42}$ are arbitrary matrices with appropriate sizes over $\mathbb{H}$, $m$ is the column number
	of $A_4$ and $n$ is the row number
	of $B_4$. 
\end{theorem}
\section*{Algorithm with a Numerical Example}

In this section, we give 
Algorithm \ref{alg1} with a numerical example to illustrate the main results.

\begin{algorithm} 
	Algorithm for computing the general solution of Equation \eqref{eq1} \label{alg1}
	
	(1)\quad Input the quaternion matrices $A_{i}, B_{i}\ (i=\overline{1,4})$ and $B$ with conformable shapes.
	
	(2)\quad Compute all matrices given by \eqref{21}--\eqref{24}.
	
	(3)\quad Check equalities in \eqref{new} or \eqref{2a}--\eqref{2i}. If not, it returns inconsistent.
	
	(4)\quad Else, compute $X_{i}\ Y_{j} (i=\overline{1,2},\ j=\overline{1,3})$.
\end{algorithm}
\begin{example} Consider the matrix Equation \eqref{eq1}. Put
	\begin{align*}
	&A_{1}=
	\begin{pmatrix}
	\mathbf{i} & 0 \\
	0 & 0 \\
	\end{pmatrix}
	,\ B_{1}=
	\begin{pmatrix}
	0 & \mathbf{i} \\
	0 & 0 \\
	\end{pmatrix}
	,\ A_{2}=
	\begin{pmatrix}
	0 & 0 \\
	\mathbf{i} & 0 \\
	\end{pmatrix}
	,\ B_{2}=
	\begin{pmatrix}
	0 & 0 \\
	0 & \mathbf{i} \\
	\end{pmatrix}
	,A_{3}=
	\begin{pmatrix}
	1 & \mathbf{i} \\
	0 & 0 \\
	\end{pmatrix}
	,\\
	&B_{3}=
	\begin{pmatrix}
	1 & \mathbf{j} \\
	0 & 0 \\
	\end{pmatrix}
	,\ A_{4}=
	\begin{pmatrix}
	1 & \mathbf{k} \\
	0 & 0 \\
	\end{pmatrix}
	,\ B_{4}=
	\begin{pmatrix}
	0 & 0 \\
	\mathbf{k} & \mathbf{i} \\
	\end{pmatrix}                                                                                                                   ,\ B=
	\begin{pmatrix}
	3\mathbf{i} & \mathbf{i}-1 \\
	0 & \mathbf{j}  
	\end{pmatrix}
	.
	\end{align*}
	Computation directly yields
	\begin{align*}
	&r
	\begin{pmatrix}
	B & A_{2} & A_{3} & A_{4} & A_{1} \\
	B_{1} & 0 & 0 & 0 & 0 \\
	\end{pmatrix}
	=r(B_{1})+r(A_{2},\ A_{3},\ A_{4},\ A_{1})=3,\\
	&r
	\begin{pmatrix}
	B & A_{2} & A_{4} & A_{1} \\
	B_{3} & 0 & 0 & 0 \\
	B_{1} & 0 & 0 & 0 \\
	\end{pmatrix}
	=r(A_{2},\ A_{4},\ A_{1})+r
	\begin{pmatrix}
	B_{3} \\
	B_{1} \\
	\end{pmatrix}
	=4,\\
	&r
	\begin{pmatrix}
	B & A_{3} & A_{4} & A_{1} \\
	B_{2} & 0 & 0 & 0 \\
	B_{1} & 0 & 0 & 0 \\
	\end{pmatrix}
	=r(A_{3},\ A_{4},\ A_{1})+r
	\begin{pmatrix}
	B_{2} \\
	B_{1} \\
	\end{pmatrix}
	=4,\\
	&r
	\begin{pmatrix}
	B & A_{4} & A_{1} \\
	B_{2} & 0 & 0 \\
	B_{3} & 0 & 0 \\
	B_{1} & 0 & 0 \\
	\end{pmatrix}
	=r
	\begin{pmatrix}
	B_{2} \\
	B_{3} \\
	B_{1} \\
	\end{pmatrix}
	+r(A_{4},\ A_{1})=3,\\
	&r
	\begin{pmatrix}
	B & A_{2} & A_{3} & A_{1} \\
	B_{4} & 0 & 0 & 0 \\
	B_{1} & 0 & 0 & 0 \\
	\end{pmatrix}
	=r(A_{2},\ A_{3},\ A_{1})+r
	\begin{pmatrix}
	B_{4} \\
	B_{1} \\
	\end{pmatrix}
	=4,\\
	&r
	\begin{pmatrix}
	B & A_{2} & A_{1} \\
	B_{3} & 0 & 0 \\
	B_{4} & 0 & 0 \\
	B_{1} & 0 & 0 \\
	\end{pmatrix}
	=r
	\begin{pmatrix}
	B_{3} \\
	B_{4} \\
	B_{1} \\
	\end{pmatrix}
	+r(A_{2},\ A_{1})=3,\\
	&r
	\begin{pmatrix}
	B & A_{3} & A_{1} \\
	B_{2} & 0 & 0 \\
	B_{4} & 0 & 0 \\
	B_{1} & 0 & 0 \\
	\end{pmatrix}
	=r
	\begin{pmatrix}
	B_{2} \\
	B_{4} \\
	B_{1} \\
	\end{pmatrix}
	+r(A_{3},\ A_{1})=3, \\
	&r
	\begin{pmatrix}
	B & A_{1} \\
	B_{2} & 0 \\
	B_{3} & 0 \\
	B_{4} & 0 \\
	B_{1} & 0 \\
	\end{pmatrix}
	=r
	\begin{pmatrix}
	B_{2} \\
	B_{3} \\
	B_{4} \\
	B_{1} \\
	\end{pmatrix}
	+r(A_{1})=3,\\
	&r
	\begin{pmatrix}
	B & A_{2} & A_{1} & 0 & 0 & 0 & A_{4} \\
	B_{3} & 0 & 0 & 0 & 0 & 0 & 0 \\
	B_{1} & 0 & 0 & 0 & 0 & 0 & 0 \\
	0 & 0 & 0 & -B & A_{3} & A_{1} & A_{4} \\
	0 & 0 & 0 & B_{2} & 0 & 0 & 0 \\
	0 & 0 & 0 & B_{1} & 0 & 0 & 0 \\
	B_{4} & 0 & 0 & B_{4} & 0 & 0 & 0 \\
	\end{pmatrix}
	=r
	\begin{pmatrix}
	B_{3} & 0 \\
	B_{1} & 0 \\
	0 & B_{2} \\
	0 & B_{1} \\
	B_{4} & B_{4} 
	\end{pmatrix}
	+r
	\begin{pmatrix}
	A_{2} & A_{1} & 0 & 0 & A_{4} \\
	0 & 0 & A_{3} & A_{1} & A_{4} \\
	\end{pmatrix}
	=7.
	\end{align*}
	All rank equalities in $\eqref{2a}$ to $\eqref{2i}$ hold. Hence, according to Theorem \ref{Thm1}, 
	Equation \eqref{eq1} has a solution. Moreover, by Theorem \ref{Thm2}, we have that
	
	\begin{align*}
	&X_{1}=
	\begin{pmatrix}
	1 & \mathbf{i} \\
	0 & 0 \\
	\end{pmatrix}
	,\ X_{2}=
	\begin{pmatrix}
	1 &\mathbf{ j} \\
	0 & 0 \\
	\end{pmatrix}
	,\ Y_{1}=
	\begin{pmatrix}
	\mathbf{i }& \mathbf{j} \\
	0 & 0 \\
	\end{pmatrix}
	,\ Y_{2}=
	\begin{pmatrix}
	\mathbf{i} & \mathbf{k} \\
	0 & 0 \\
	\end{pmatrix}
	,\ Y_{3}=
	\begin{pmatrix}
	\mathbf{i} & \mathbf{j} \\
	\mathbf{k} & 0 \\
	\end{pmatrix}.	
	\end{align*}
\end{example}

\begin{remark} 
	Chu et al. gave potential applications of the maximal and minimal ranks in the discipline of control theory (e.g., \cite{D.L. 1998,D.L. 2000,D.L. 2009}). We may consider the rank bounds of the general solution of 
	Equation \eqref{eq1}.
\end{remark}

\section{The General Solution to equation with \boldmath{$\eta$}-Hermicity \label{sec4}}

In this section, as an application of \eqref{eq1}, we establish some necessary and sufficient conditions for quaternion matrix Equation \eqref{eq5} to have a solution and derive a formula of its general solution involving $\eta$-Hermicity.
\begin{theorem}\label{the4.1} Let $ A_{i}$ $(i=\overline{1,4})$ and $B$  be given matrices with suitable sizes over $\mathbb{H}$, $B=B^{\eta^{\ast}}$. Set\begingroup\makeatletter\def\f@size{9}\check@mathfonts
	\def\maketag@@@#1{\hbox{\m@th\normalsize \normalfont#1}}%
	\begin{align*}
	&R_{A_{1}}A_{2}=A_{11},\ R_{A_{1}}A_{3}=A_{22},\ R_{A_{1}}A_{4}=A_{33},\ R_{A_{11}}A_{22}=M_{1},\ S_{1}=A_{22}L_{M_{1}},\\
	&R_{A_{1}}B(R_{A_{1}})^{\eta^{\ast}}=T_{1},\  C=R_{M_{1}}R_{A_{11}},\ C_{1}=CA_{33},\ C_{2}=R_{A_{11}}A_{33},\\
	&C_{3}=R_{A_{22}}A_{33},\ C_{4}=A_{33},\ E_{1}=CT_{1},\ E_{2}=R_{A_{11}}T_{1}(R_{A_{22}})^{\eta^{\ast}},\ E_{3}=R_{A_{22}}T_{1}(R_{A_{11}})^{\eta^{\ast}},\ E_{4}=T_{1}C^{\eta^{\ast}},\\
	&C_{11}=(L_{C_{2}},\ L_{C_{4}}),\ C_{22}=L_{C_{1}},\ C_{33}=L_{C_{3}},\ E_{11}=R_{C_{11}}C_{22},\ E_{22}=R_{C_{11}}C_{33},\\
	&M=R_{E_{11}}E_{22},\ N=(R_{E_{22}}E_{11})^{\eta^{\ast}},\ F=F_{2}-F_{1},\ E=R_{C_{11}}F(R_{C_{11}})^{\eta^{\ast}},\ S=E_{22}L_{M},\\
	&F_{11}=C_{2}L_{C_{1}},\ G_{1}=E_{2}-C_{2}C_{1}^{\dagger}E_{1}(C_{4}^{\eta^{\ast}})^{\dagger}C_{3}^{\eta^{\ast}},\ F_{22}=C_{4}L_{C_{3}},\ G_{2}=E_{4}-C_{4}C_{3}^{\dagger}E_{3}(C_{2}^{\eta^{\ast}})^{\dagger}C_{1}^{\eta^{\ast}},\\
	&F_{1}=C_{1}^{\dagger}E_{1}(C_{4}^{\eta^{\ast}})^{\dagger}+L_{C_{1}}C_{2}^{\dagger}E_{2}(C_{3}^{\eta^{\ast}})^{\dagger},\ F_{2}=C_{3}^{\dagger}E_{3}(C_{2}^{\eta^{\ast}})^{\dagger}+L_{C_{3}}C_{4}^{\dagger}E_{4}(C_{1}^{\eta^{\ast}})^{\dagger}.
	\end{align*}
	\endgroup
	Then, the following statements are equivalent:
	
	$\mathrm{(1)}$ Equation \eqref{eq5}  is consistent.
	
	$\mathrm{(2)}$  $R_{C_{i}}E_{i}=0$  $(i=\overline{1,4})$,\ $R_{E_{22}}E(R_{E_{22}})^{\eta^{\ast}}=0$.
	
	$\mathrm{(3)}$
	\begin{align*}
	&r
	\begin{pmatrix}
	B & A_{2} & A_{3} & A_{4} & A_{1} \\
	A_{1}^{\eta^{\ast}} & 0 & 0 & 0 & 0 \\
	\end{pmatrix}
	=r(A_{1})+r(A_{2},\ A_{3},\ A_{4},\ A_{1}),\\
	&r
	\begin{pmatrix}
	B & A_{2} & A_{3} & A_{1} \\
	A_{4}^{\eta^{\ast}} & 0 & 0 & 0 \\
	A_{1}^{\eta^{\ast}} & 0 & 0 & 0 \\
	\end{pmatrix}
	=r(A_{2},\ A_{3},\ A_{1})+r(A_{4},\ A_{1}),\\
	&r
	\begin{pmatrix}
	B & A_{2} & A_{4} & A_{1} \\
	A_{3}^{\eta^{\ast}} & 0 & 0 & 0 \\
	A_{1}^{\eta^{\ast}} & 0 & 0 & 0 \\
	\end{pmatrix}
	=r(A_{2},\ A_{4},\ A_{1})+r(A_{3},\ A_{1}),\\
	&r
	\begin{pmatrix}
	B & A_{3} & A_{4} & A_{1} \\
	A_{2}^{\eta^{\ast}} & 0 & 0 & 0 \\
	A_{1}^{\eta^{\ast}} & 0 & 0 & 0 \\
	\end{pmatrix}
	=r(A_{3},\ A_{4},\ A_{1})+r(A_{2},\ A_{1}),
	\end{align*}
	\begin{align*}
	&r
	\begin{pmatrix}
	B & 0 & A_{2} & 0 & A_{4} & A_{1} & 0 \\
	0 & -B & 0 & A_{3} & A_{4} & 0 & A_{1} \\
	A_{3}^{\eta^{\ast}} & 0 & 0 & 0 & 0 & 0 & 0 \\
	0 & A_{2}^{\eta^{\ast}} & 0 & 0 & 0 & 0 & 0 \\
	A_{4}^{\eta^{\ast}} & A_{4}^{\eta^{\ast}} & 0 & 0 & 0 & 0 & 0 \\
	A_{1}^{\eta^{\ast}} & 0 & 0 & 0 & 0 & 0 & 0 \\
	0 & A_{1}^{\eta^{\ast}} & 0 & 0 & 0 & 0 & 0 \\
	\end{pmatrix}	
	=2r\begin{pmatrix}
	A_{2} & 0 & A_{4} & A_{1} & 0 \\
	0 & A_{3} & A_{4} & 0 & A_{1} \\
	\end{pmatrix}.
	\end{align*}
	In this case, the general solution to Equation \eqref{eq5} can be expressed as\begingroup\makeatletter\def\f@size{9}\check@mathfonts
	\def\maketag@@@#1{\hbox{\m@th\normalsize \normalfont#1}}%
	\begin{align*}
	&X_{1}=\frac{\widehat{X_{1}}+(\widehat{X_{2}})^{\eta^{\ast}}}{2},\ Y_{1}=\frac{\widehat{Y_{1}}+(\widehat{Y_{1}})^{\eta^{\ast}}}{2},\ Y_{2}=\frac{\widehat{Y_{2}}+(\widehat{Y_{2}})^{\eta^{\ast}}}{2},\ Y_{3}=\frac{\widehat{Y_{3}}+(\widehat{Y_{3}})^{\eta^{\ast}}}{2},\\
	&\widehat{X_{1}}=A_{1}^{\dagger}(C_{1}-A_{2}Y_{1}A_{2}^{\eta^{\ast}}-A_{3}Y_{2}A_{3}^{\eta^{\ast}}-A_{4}Y_{3}A_{4}^{\eta^{\ast}})+L_{A_{1}}U_{2},\\
	&\widehat{X_{2}}=R_{A_{1}}(C_{1}-A_{2}Y_{1}A_{2}^{\eta^{\ast}}-A_{3}Y_{2}A_{3}^{\eta^{\ast}}-A_{4}Y_{3}A_{4}^{\eta^{\ast}})(A_{1}^{\dagger})^{\eta^{\ast}}+A_{1}A_{1}^{\dagger}U_{1}+U_{3}R_{A_{1}^{\eta^{\ast}}},\\
	&\widehat{Y_{1}}=A_{11}^{\dagger}T(A_{11}^{\dagger})^{\eta^{\ast}}-A_{11}^{\dagger}A_{22}M_{1}^{\dagger}T(A_{11}^{\dagger})^{\eta^{\ast}}-A_{11}^{\dagger}U_{4}A_{22}^{\dagger}T(M_{1}^{\dagger})^{\eta^{\ast}}(A_{22}^{\dagger})^{\eta^{\ast}}+L_{A_{11}}U_{5}+U_{6}R_{A_{11}^{\eta^{\ast}}},\\
	&\widehat{Y_{2}}=M_{1}^{\dagger}T(A_{22}^{\dagger})^{\eta^{\ast}}+S_{1}^{\dagger}S_{1}A_{22}^{\dagger}T(M_{1}^{\dagger})^{\eta^{\ast}}+L_{M_{1}}L_{S_{1}}U_{7}+U_{8}R_{A_{22}^{\eta^{\ast}}}+L_{M_{1}}U_{4}R_{M_{1}^{\eta^{\ast}}},\\
	&\widehat{Y_{3}}=F_{1}+L_{C_{2}}V_{1}+V_{2}R_{C_{4}^{\eta^{\ast}}}+L_{C_{1}}V_{3}R_{C_{3}^{\eta^{\ast}}},\ or\ \widehat{Y_{3}}=F_{2}-L_{C_{4}}W_{1}-W_{2}R_{C_{2}^{\eta^{\ast}}}-L_{C_{3}}W_{3}R_{C_{1}^{\eta^{\ast}}},
	\end{align*}
	\endgroup
	where $T=T_{1}-A_{33}Y_{3}(A_{33})^{\eta^{\ast}}$,
	\begin{align*}
	&V_{1}=(I_{m},\ 0)\left[C_{11}^{\dagger}(F-C_{22}V_{3}C_{33}^{\eta^{\ast}}-C_{33}W_{3}C_{22}^{\eta^{\ast}})-C_{11}^{\dagger}U_{11}C_{11}^{\eta^{\ast}}+L_{C_{11}}U_{12}\right],\\
	&W_{1}=(0,\ I_{m})\left[C_{11}^{\dagger}(F-C_{22}V_{3}C_{33}^{\eta^{\ast}}-C_{33}W_{3}C_{22}^{\eta^{\ast}})-C_{11}^{\dagger}U_{11}C_{11}^{\eta^{\ast}}+L_{C_{11}}U_{12}\right],\\
	&W_{2}=\left[R_{C_{11}}(F-C_{22}V_{3}C_{33}^{\eta^{\ast}}-C_{33}W_{3}C_{22}^{\eta^{\ast}})(C_{11}^{\eta^{\ast}})^{\dagger}+C_{11}C_{11}^{\dagger}U_{11}+U_{21}L_{C_{11}}^{\eta^{\ast}}\right]
	\begin{pmatrix}
	0 \\
	I_{n} \\
	\end{pmatrix}
	,\\
	&V_{2}=\left[R_{C_{11}}(F-C_{22}V_{3}C_{33}^{\eta^{\ast}}-C_{33}W_{3}C_{22}^{\eta^{\ast}})(C_{11}^{\eta^{\ast}})^{\dagger}+C_{11}C_{11}^{\dagger}U_{11}+U_{21}L_{C_{11}}^{\eta^{\ast}}\right]
	\begin{pmatrix}
	I_{n} \\
	0 \\
	\end{pmatrix},\\
	&V_{3}=E_{11}^{\dagger}F(E_{22}^{\eta^{\ast}})^{\dagger}-E_{11}^{\dagger}E_{22}M^{\dagger}F(E_{22}^{\eta^{\ast}})^{\dagger}-E_{11}^{\dagger}SE_{22}^{\dagger}FN^{\dagger}E_{11}^{\eta^{\ast}}(E_{22}^{\eta^{\ast}})^{\dagger}\\
	&-E_{11}^{\dagger}SU_{31}R_{N}E_{11}^{\eta^{\ast}}(E_{22}^{\eta^{\ast}})^{\dagger}+L_{E_{11}}U_{32}+U_{33}L_{E_{22}}^{\eta^{\ast}},\\
	&W_{3}=M^{\dagger}F(E_{11}^{\eta^{\ast}})^{\dagger}+S^{\dagger}SE_{22}^{\dagger}FN^{\dagger}+L_{M}L_{S}U_{41}+L_{M}U_{31}R_{N}-U_{42}L_{E_{11}}^{\eta^{\ast}},\quad\quad\quad\quad\quad\quad\quad\quad\quad
	\end{align*}
	$U_{11}, U_{12}$, $U_{21}$, $U_{31}$, $U_{32}$, $U_{33}$, $U_{41}$,
	and $U_{42}$ are any matrices with suitable dimensions over $\mathbb{H}$.
\end{theorem}
\begin{proof}
	It is easy to show that \eqref{eq5}  has a solution if and only if the following matrix equation has a solution:
	\begin{equation}\label{eq4.2}
	\begin{aligned}
	A_{1}\widehat{X_{1}}+\widehat{X_{2}}A_{1}^{\eta^{\ast}}+A_{2}\widehat{Y_{1}}A_{2}^{\eta^{\ast}}+A_{3}\widehat{Y_{2}}A_{3}^{\eta^{\ast}}+A_{4}\widehat{Y_{3}}A_{4}^{\eta^{\ast}}=B.
	\end{aligned}
	\end{equation}
	If \eqref{eq5}  has a solution,  say,
	$(X_{1},\ Y_{1},\ Y_{2},\ Y_{3})$, then
	$$
	(\widehat{X_{1}},\ \widehat{X_{2}},\ \widehat{Y_{1}},\ \widehat{Y_{2}},\ \widehat{Y_{3}}):=(X_{1},\ X_{1}^{\eta^{\ast}},\ Y_{1},\ Y_{2},\ Y_{3})
	$$
	is a solution of \eqref{eq4.2}. 
	Conversely, if \eqref{eq4.2}  has a solution, say
	$$
	(\widehat{X_{1}},\ \widehat{X_{2}},\ \widehat{Y_{1}},\ \widehat{Y_{2}},\ \widehat{Y_{3}}).
	$$ 
	It is easy to show that \eqref{eq5} has a solution 
	\begin{align*}
	(X_{1},\ Y_{1},\ Y_{2},\ Y_{3}):=\left( \frac{\widehat{X_{1}}+(\widehat{X_{2}})^{\eta^{\ast}}}{2},\ \frac{\widehat{Y_{1}}+(\widehat{Y_{1}})^{\eta^{\ast}}}{2},\ \frac{\widehat{Y_{2}}+(\widehat{Y_{2}})^{\eta^{\ast}}}{2},\ \frac{\widehat{Y_{3}}+(\widehat{Y_{3}})^{\eta^{\ast}}}{2}\right) .
	\end{align*}\end{proof}
Letting $A_{1}$ and $B_{1}$ vanish in Theorem \ref{Thm1}, it  yields to the
following result. 

\begin{corollary}
	Let $A_{ii},$ $B_{ii}$  $(i=\overline{1,3})$,
	and $T_{1}$ be given matrices with appropriate sizes over $\mathbb{H}$. Set
	\begin{align*}
	&M_{1}=R_{A_{11}}A_{22},\ N_{1}=B_{22}L_{B_{11}},\ S_{1}=A_{22}L_{M_{1}},\\
	&C=R_{M_{1}}R_{A_{11}},\ C_{1}=CA_{33},\ C_{2}=R_{A_{11}}A_{33},\ C_{3}=R_{A_{22}}A_{33},\ C_{4}=A_{33},\\
	&D=L_{B_{11}}L_{N_{1}},\ D_{1}=B_{33},\ D_{2}=B_{33}L_{B_{22}},\ D_{3}=B_{33}L_{B_{11}},\ D_{4}=B_{33}D,\\
	&E_{1}=CT_{1},\ E_{2}=R_{A_{11}}T_{1}L_{B_{22}},\ E_{3}=R_{A_{22}}T_{1}L_{B_{11}},\ E_{4}=T_{1}D,\\
	&C_{11}=(L_{C_{2}},\ L_{C_{4}}),\ D_{11}=
	\begin{pmatrix}
	R_{D_{1}} \\
	R_{D_{3}} \\
	\end{pmatrix}
	,\ C_{22}=L_{C_{1}},\ D_{22}=R_{D_{2}},\ C_{33}=L_{C_{3}},\\
	&D_{33}=R_{D_{4}},\ E_{11}=R_{C_{11}}C_{22},\ E_{22}=R_{C_{11}}C_{33},\ E_{33}=D_{22}L_{D_{11}},\ E_{44}=D_{33}L_{D_{11}},\\
	&M=R_{E_{11}}E_{22},\ N=E_{44}L_{E_{33}},\ F=F_{2}-F_{1},\ E=R_{C_{11}}FL_{D_{11}},\ S=E_{22}L_{M},\\
	&F_{11}=C_{2}L_{C_{1}},\ G_{1}=E_{2}-C_{2}C_{1}^{\dagger}E_{1}D_{1}^{\dagger}D_{2},\ F_{22}=C_{4}L_{C_{3}},\  G_{2}=E_{4}-C_{4}C_{3}^{\dagger}E_{3}D_{3}^{\dagger}D_{4},\\
	&F_{1}=C_{1}^{\dagger}E_{1}D_{1}^{\dagger}+L_{C_{1}}C_{2}^{\dagger}E_{2}D_{2}^{\dagger},\ F_{2}=C_{3}^{\dagger}E_{3}D_{3}^{\dagger}+L_{C_{3}}C_{4}^{\dagger}E_{4}D_{4}^{\dagger}.
	\end{align*}
	Then, the following statements are equivalent:
	
	$\mathrm{(1)}$ Equation \eqref{eq6}  is consistent.
	
	$\mathrm{(2)}$ $R_{C_{i}}$$E_{i}=0$,\ $E_{i}$$L_{D_{i}}=0$ $(i=\overline{1,4})$,\ $R_{E_{22}}EL_{E_{33}}=0$.
	
	$\mathrm{(3)}$
	\begin{align*}
	&r(T_{1},\ A_{11},\ A_{22},\ A_{33})=r(A_{11},\ A_{22},\ A_{33}),\\
	& r
	\begin{pmatrix}
	T_{1} \\
	B_{11} \\
	B_{22} \\
	B_{33} \\
	\end{pmatrix}
	=r
	\begin{pmatrix}
	B_{11} \\
	B_{22} \\
	B_{33} \\
	\end{pmatrix}
	,\ r
	\begin{pmatrix}
	T_{1} & A_{11} & A_{22} \\
	B_{33} & 0 & 0 \\
	\end{pmatrix}
	=r(A_{11},\ A_{22})+r(B_{33}),\\
	& r
	\begin{pmatrix}
	T_{1} & A_{11} & A_{33} \\
	B_{22} & 0 & 0 \\
	\end{pmatrix}
	=r(A_{11},\ A_{33})+r(B_{22}),\\
	&r
	\begin{pmatrix}
	T_{1} & A_{33} & A_{22} \\
	B_{11} & 0 & 0 \\
	\end{pmatrix}
	=r(A_{33},\ A_{22})+r(B_{11}),\ r
	\begin{pmatrix}
	T_{1} & A_{33} \\
	B_{11} & 0 \\
	B_{22} & 0 \\
	\end{pmatrix}
	=r
	\begin{pmatrix}
	B_{11} \\
	B_{22} \\
	\end{pmatrix}
	+r(A_{33}),\\
	\end{align*}
	\begin{align*}
	\footnotesize
	&r
	\begin{pmatrix}
	T_{1} & 0 & A_{11} & 0 & A_{33} \\
	0 & -T_{1} & 0 & A_{22} & A_{33} \\
	B_{22} & 0 & 0 & 0 & 0 \\
	0 & B_{11} & 0 & 0 & 0 \\
	B_{33} & B_{33} & 0 & 0 & 0 \\
	\end{pmatrix}
	=r
	\begin{pmatrix}
	B_{22} & 0 \\
	0 & B_{11} \\
	B_{33} & B_{33} \\
	\end{pmatrix}
	+r
	\begin{pmatrix}
	A_{11} & 0 & A_{33} \\
	0 & A_{22} & A_{33} \\
	\end{pmatrix},\\
	&r
	\begin{pmatrix}
	T_{1} & A_{22} \\
	B_{11} & 0 \\
	B_{33} & 0 \\
	\end{pmatrix}
	=r
	\begin{pmatrix}
	B_{11} \\
	B_{33} \\
	\end{pmatrix}
	+r(A_{22}),\ r
	\begin{pmatrix}
	T_{1} & A_{11} \\
	B_{33} & 0 \\
	B_{22} & 0 \\
	\end{pmatrix}
	=r
	\begin{pmatrix}
	B_{33} \\
	B_{22} \\
	\end{pmatrix}
	+r(A_{11}).
	\end{align*}
	In this case, the general solution to Equation \eqref{eq6} can be expressed as
	\begin{align*}
	&Y_{1}=A_{11}^{\dagger}TB_{11}^{\dagger}-A_{11}^{\dagger}A_{22}M_{1}^{\dagger}TB_{11}^{\dagger}-A_{11}^{\dagger}S_{1}A_{22}^{\dagger}TN_{1}^{\dagger}B_{22}B_{11}^{\dagger}\\
	&-A_{11}^{\dagger}S_{1}U_{4}R_{N_{1}}B_{22}B_{11}^{\dagger}+L_{A_{11}}U_{5}+U_{6}R_{B_{11}},\\
	&Y_{2}=M_{1}^{\dagger}TB_{22}^{\dagger}+S_{1}^{\dagger}S_{1}A_{22}^{\dagger}TN_{1}^{\dagger}+L_{M_{1}}L_{S_{1}}U_{7}+U_{8}R_{B_{22}}+L_{M_{1}}U_{4}R_{N_{1}},\\
	&Y_{3}=F_{1}+L_{C_{2}}V_{1}+V_{2}R_{D_{1}}+L_{C_{1}}V_{3}R_{D_{2}},\ or \ Y_{3}=F_{2}-L_{C_{4}}W_{1}-W_{2}R_{D_{3}}-L_{C_{3}}W_{3}R_{D_{4}},
	\end{align*}
	where $T=T_{1}-A_{33}Y_{3}B_{33}$, $U_{i}(i=1,...,8)$ are any matrices with suitable dimensions over $\mathbb{H}$,\begingroup\makeatletter\def\f@size{9}\check@mathfonts
	\def\maketag@@@#1{\hbox{\m@th\normalsize \normalfont#1}}%
	\begin{align*}
	&V_{1}=(I_{m},\ 0)\left[C_{11}^{\dagger}(F-C_{22}V_{3}D_{22}-C_{33}W_{3}D_{33})-C_{11}^{\dagger}U_{11}D_{11}+L_{C_{11}}U_{12}\right],\\
	&W_{1}=(0,\ I_{m})\left[C_{11}^{\dagger}(F-C_{22}V_{3}D_{22}-C_{33}W_{3}D_{33})-C_{11}^{\dagger}U_{11}D_{11}+L_{C_{11}}U_{12}\right],\\
	&W_{2}=\left[R_{C_{11}}(F-C_{22}V_{3}D_{22}-C_{33}W_{3}D_{33})D_{11}^{\dagger}+C_{11}C_{11}^{\dagger}U_{11}+U_{21}R_{D_{11}}\right]
	\begin{pmatrix}
	0 \\
	I_{n} \\
	\end{pmatrix},\\
	&V_{2}=\left[R_{C_{11}}(F-C_{22}V_{3}D_{22}-C_{33}W_{3}D_{33})D_{11}^{\dagger}+C_{11}C_{11}^{\dagger}U_{11}+U_{21}R_{D_{11}}\right]
	\begin{pmatrix}
	I_{n} \\
	0 \\
	\end{pmatrix}
	,\\
	&V_{3}=E_{11}^{\dagger}FE_{33}^{\dagger}-E_{11}^{\dagger}E_{22}M^{\dagger}FE_{33}^{\dagger}-E_{11}^{\dagger}SE_{22}^{\dagger}FN^{\dagger}E_{44}E_{33}^{\dagger}-E_{11}^{\dagger}SU_{31}R_{N}E_{44}E_{33}^{\dagger}+L_{E_{11}}U_{32}+U_{33}R_{E_{33}},\\
	&W_{3}=M^{\dagger}FE_{44}^{\dagger}+S^{\dagger}SE_{22}^{\dagger}FN^{\dagger}+L_{M}L_{S}U_{41}+L_{M}U_{31}R_{N}-U_{42}R_{E_{44}},
	\end{align*}
	$U_{11}, U_{12}$, $U_{21}$, $U_{31}$, $U_{32}$, $U_{33}$, $U_{41}$,
	and $U_{42}$ are any matrices with suitable dimensions over $\mathbb{H}$.
	
	\endgroup
\end{corollary}
\section{Conclusions \label{sec5}}

We have established the solvability conditions and an
exact formula of a general solution to  quaternion matrix Equation \eqref{eq1}. As an application of Equation \eqref{eq1}, we also have established some necessary and sufficient conditions for Equation \eqref{eq5} to have a solution and derived a formula of its general solution involving
$\eta$-Hermicity. The quaternion matrix Equation \eqref{eq1} plays a key role in studying the solvability conditions and general solutions of other types of matrix equations.
For example, we can use the results on Equation \eqref{eq1} to investigate the solvability conditions and the general solution of the following system of quaternion matrix equations
\begin{align*}
&A_{2}Y_{1}=C_{2},\ Y_{1}B_{2}=D_{2},\\
&A_{3}Y_{2}=C_{3},\ Y_{2}B_{3}=D_{3},\\
&A_{4}Y_{3}=C_{4},\ Y_{3}B_{4}=D_{4},\\
G_{1}Y_{1}&H_{1}+G_{2}Y_{2}H_{2}+G_{3}Y_{3}H_{3}=G
\end{align*}
where  $Y_{1}, Y_{2}$ and $Y_{3}$ are unknown quaternion matrices and the others are given. 

It is worth mentioning
that the main results of \eqref{eq1}  are available over not only  
$\mathbb{R}$ and $\mathbb{C}$ but also any division ring.  
Moreover, inspired by \cite{Li 2022}, we can investigate 
Equation \eqref{eq1} in tensor form.


\begin{thebibliography}{99}



\bibitem{C.C 2011} Took, C.C.; Mandic, D.P. Augmented second-order statistics of quaternion random signals.\emph{ Signal Process.} \textbf{2011}, \emph{91}, 214--224.


\bibitem{L.Q. 2021}  Qi, L.; Luo, Z.Y.; Wang, Q.W.; Zhang, X.Z. Quaternion matrix optimization: Motivation and analysis. J. Optim. Theory Appl. \textbf{2021}, 193,  621-648. https://doi.org/10.1007/s10957-021-01906-y.


\bibitem{Jia 2018} Jia, Z.G.; Ling, S.T.; Zhao, M.X. Color two-dimensional principal component analysis for face recognition based on quaternion model. In  {Proceedings of the }International Conference on Intelligent Computing: Intelligent Computing Theories and Application,   Liverpool, UK,  7--10 August {2017}; pp. 177--189.

\bibitem {Wang X.X.} Wang, Q.W.; Wang, X.X. Arnoldi method for large quaternion right eigenvalue problem. \emph{J. Sci. Comput. }\textbf{2020}, \emph{58}, 1--20.

\bibitem {Shahzad} Shahzad, A.; Jones, B.L.; Kerrigan, E.C.; Constantinides, G.A. An efficient algorithm for the solution of a coupled sylvester equation appearing in descriptor systems. \emph{Automatica}  \textbf{2011}, \emph{47}, 24--48.

\bibitem {Syrmosw33}  Syrmos, V.L.; Lewis, F.L. Coupled and constrained Sylvester equations in system design. \emph{Circuits Syst. Signal Process.} \textbf{1994}, \emph{13}, 66--94.


\bibitem {Liw33}  Li, R.C. A bound on the solution to a structured Sylvester equation with an application to relative perturbation theory. \emph{SIAM J. Matrix Anal. Appl. }\textbf{1999}, \emph{21}, 44--45.	


\bibitem{Barraud} Barraud, A.;  Lesecq, S.; Christov, N. From sensitivity analysis to random floating point arithmetics-application to Sylvester equations. In  {Proceedings of the } International Conference on Numerical Analysis and Its Applications,  Rousse, Bulgaria,  11--15  June 2000; {Volume {1998}}; p. 351.

\bibitem {Saberiw33} Saberi, A.; Stoorvogel, A.A.; Sannuti, P. \emph{Control of Linear Systems with Regulation and Input Constraints}; Springer:     {Berlin/Heidelberg, Germany,} 
{2003}.

\bibitem{Darouach}Darouach, M. Solution to Sylvester equation associated to linear descriptor systems. \emph{Syst. Control Lett.}  {\textbf{2006}}, \emph{55}, 835--838.	


\bibitem{E.B. 2005}  Castelan,E.B.; Gomes da Silva, V. On the solution of a Sylvester matrix equation appearing in descriptor systems control theory. \emph{Syst. Control Lett.} \textbf{2005}, \emph{54}, 109--117.


\bibitem{W.E. 1952} Roth, W.E. The equations $AX-YB=C$ and $AX-XB=C$ in matrices. \emph{Proc. Am. Math. Soc.} \textbf{1952}, \emph{3}, 392--396.

\bibitem{J.K. 1979} Baksalary, J.K.;  Kala, R. The matrix equations $AX-YB=C$. \emph{Linear Algebra Appl.} \textbf{1979}, \emph{25}, 41--43.


\bibitem{J.K. 1980} Baksalary, J.K.;  Kala, R. The matrix equations $AXB+CYD=E$. \emph{Linear Algebra Appl.} \textbf{1979}, \emph{30}, 141--147.

\bibitem {A.B. 1991}   \"{O}zg\"{u}ler, A.B. The matrix equation $AXB+CYD=E$ over a principal ideal domain. \emph{SIAM J. Matrix Anal. Appl.} \textbf{1991}, \emph{12}, 581--591.
%



\bibitem {T28} Wang, Q.W. A system of matrix equations and a linear matrix equation over arbitrary regular ring with identity. \emph{Linear Algebra Appl.} \textbf{2004}, \emph{384}, 43--54.


\bibitem {xinliu01} Liu, X. The $\eta$-anti-Hermitian solution to some classic matrix equations. \emph{Appl. Math. Comput.} \textbf{2018}, \emph{320}, 264--270.

\bibitem {xinliu03}	Liu, X.; Zhang, Y. Consistency of split quaternion matrix equations $AX^{\star}-XB=CY+D$ and $X-AX^{\star}B=CY+D$. \emph{Adv. Appl. Clifford Algebras} \textbf{2019}, \emph{64}, 1--20.	


\bibitem {xinliu04}	 Liu, X.; Song, G.J.; Zhang, Y. Determinantal representations of the solutions to systems of generalized sylvester equations. \emph{Adv. Appl. Clifford Algebras }\textbf{2019}, \emph{12}, 1--19.


\bibitem{WangThree}   Mehany, M.S.;  Wang, Q.W.  Three symmetrical systems of coupled Sylvester-like quaternion matrix equations. \emph{Symmetry}  \textbf{2022}, \emph{14}, 550. https://doi.org/10.3390/sym14030550.


\bibitem{Jiang 2022} Jiang, J.; Li, N. An iterative algorithm for the generalized reflexive solution group of a system of quaternion matrix equations. \emph{Symmetry} \textbf{2022}, \emph{14}, 776.
\bibitem {Liu} Liu, L.S.; Wang, Q.W.; Chen, J.F.; Xie, Y.Z. An exact solution to a quaternion matrix equation with an application. \emph{Symmetry}  \textbf{2022}, \emph{14}, 375.



\bibitem{Wang 2016}  Wang, Q.W.; Rehman, A.; He, Z.H.; Zhang, Y. Constrained generalized Sylvester matrix equations.  \emph{Automatica} \textbf{2016}, \emph{69}, 60--64.

\bibitem{Wang 2019}  Wang, Q.W.; He, Z.H.; Zhang, Y. Constrained two-sided coupled Sylvester-type quaternion matrix equations. \emph{Automatica} \textbf{2019}, \emph{101}, 207--213.

\bibitem{Rodman} Rodman, L.  \emph{Topics in Quaternion Linear Algebra}; Princeton University Press: Princeton, NJ, USA,   {2014}.

\bibitem{Jia 2019} Jia, Z.G.; Ng, M.K.; Song, G.J. Robust quaternion matrix completion with applications to image inpainting. \emph{Numer. Linear Algebra Appl.} \textbf{2019}, \emph{26}, 1--35.



\bibitem {He 2021} Yu, S.W.; He, Z.H.; Qi, T.C.; Wang, X.X. The equivalence canonical form of five quaternion matrices with applications to imaging and Sylvester-type equations. \emph{J. Comput. Appl. Math. } \textbf{2021}, \emph{393}, 113494.


\bibitem {S.F. 2013} Yuan, S.F.; Wang, Q.W.; Duan, X.F. On solutions of the quaternion matrix equation $AX=B$ and their applications in color image restoration. \emph{J. Comput. Appl. Math.}  \textbf{2013}, \emph{221}, 10--20.	


\bibitem{He 2019}  He, Z.H. Some new results on a system of Sylvester-type quaternion matrix equations. \emph{Linear Multilinear Algebra } \textbf{2021}, \emph{69}, 3069--3091.

\bibitem {Kyrchei0112}  Kyrchei, I. Cramers rules for Sylvester quaternion matrix equation and its special cases. \emph{Adv. Appl. Clifford Algebras} \textbf{2018}, \emph{28}, 1--26.


\bibitem{Wang 2012} Wang, Q.W.; He, Z.H. Some matrix equations with applications. \emph{Linear  Multilinear Algebra} \textbf{2012}, \emph{60}, 1327--1353.


\bibitem {wangronghao} Zhang, Y.; Wang, R.H. The exact solution of a system of quaternion matrix equations involving $\eta$-Hermicity. \emph{Appl. Math. Comput. }\textbf{2013}, \emph{222}, 201--209.




\bibitem{F.Z 2011}  Took, C.C.; Mandic, D.P.; Zhang, F.Z. On the unitary diagonalization of a special class of quaternion matrices. \emph{Appl. Math. Lett. } \textbf{2011}, \emph{24}, 1806--1809.

\bibitem{M.G 1974} Marsaglia, G.; Styan, G.P.  Equalities and inequalities for ranks of matrices. \emph{Linear Multilinear Algebra} \textbf{1974}, \emph{2}, 269--292.











\bibitem{Z.H. 2013}  He, Z.H.; Wang, Q.W. A real quaternion matrix equation with applications. \emph{Linear  Multilinear Algebra} \textbf{2013}, \emph{61}, 725--740.



\bibitem{D.L. 1998}  Chu, D.L.; Chan, H.; Ho, D.W.C. 	Regularization of singular systems by derivative and proportional output feedback. \emph{SIAM J. Math. Anal.} \textbf{1998}, \emph{19}, 21--38.




\bibitem{D.L. 2000}  Chu, D.L.; De Lathauwer, L.; Moor, B. On the computation of restricted singular value decomposition via cosine-sine decomposition. \emph{SIAM J. Math. Anal.} \textbf{2000}, \emph{22}, 550--601.

\bibitem{D.L. 2009} Chu, D.L.;  Hung, Y.S.; Woerdeman, H.J. Inertia and rank characterizations of some matrix expressions. \emph{SIAM J. Math. Anal.} \textbf{2009}, \emph{31}, 1187--1226.





\bibitem {Li 2022} Li, T.; Wang, Q.W.; Zhang, X.F. A Modified conjugate residual method and nearest kronecker product preconditioner for the generalized coupled Sylvester tensor equations.  \emph{Mathematics}  \textbf{2022}, \emph{10}, 1730.

	
	
	
\end{thebibliography}
\end{document}